\documentclass[12pt]{amsart}
\usepackage[margin=1.1in]{geometry}
\usepackage{amssymb,amsfonts,amsmath,mathrsfs,bbm}
\usepackage[unicode]{hyperref}
\usepackage[capitalise]{cleveref}
\usepackage[shortlabels]{enumitem}
\hypersetup{colorlinks=true, citecolor=blue, linkcolor=blue, urlcolor=blue, pdfstartview=FitH, pdfauthor=Jesse Thorner, pdftitle=Power sums and Siegel-type zero-free regions for $L$-functions}

\newtheorem{theorem}{Theorem}[section]
\newtheorem{proposition}[theorem]{Proposition}
\newtheorem{lemma}[theorem]{Lemma}

\theoremstyle{remark}
\newtheorem*{remark}{Remark}

\numberwithin{equation}{section}

\title{Power sums and Siegel-type zero-free regions for $L$-functions}

\author{Jesse Thorner}
\address{Department of Mathematics, University of Illinois Urbana-Champaign, 1409 West
Green Street, Urbana, IL 61801, USA}
\email{jesse.thorner@gmail.com}

\begin{document}

\begin{abstract}
Let $\pi$ and $\pi'$ be unitary cuspidal automorphic representations of $\mathrm{GL}(n)$ and $\mathrm{GL}(n')$ over a number field $F$.  Let $\mathfrak{C}_{\pi}$ be the analytic conductor of $\pi$.  We develop a new approach to zero-free regions for $L$-functions via lower bounds for power sums, proving for all $\varepsilon>0$ the existence of ineffective constants $c=c_{n,F,\varepsilon}>0$ and $c'=c'_{n,F,\pi',\varepsilon}>0$ such that the standard $L$-function $L(s,\pi)$ satisfies
\[
|L(\sigma+it,\pi)|\geq c(\mathfrak{C}_{\pi}(|t|+3))^{-\varepsilon},\qquad \sigma\geq 1-c(\mathfrak{C}_{\pi}(|t|+3))^{-\varepsilon}
\]
and the Rankin--Selberg $L$-function $L(s,\pi\times\pi')$ satisfies
\[
|L(\sigma+it,\pi\times\pi')|\geq c'(\mathfrak{C}_{\pi}(|t|+3))^{-\varepsilon},\qquad \sigma\geq 1-c'(\mathfrak{C}_{\pi}(|t|+3))^{-\varepsilon}.
\]
Applications include improvements to the prime number theorems for these $L$-functions and new generalizations of the Brauer--Siegel theorem.
\end{abstract}

\maketitle


\section{Introduction and main results}
\label{sec:intro}

We develop a new method for establishing zero-free regions near $\mathrm{Re}(s)=1$ for $\mathrm{GL}(n)$ standard $L$-functions $L(s,\pi)$ and $\mathrm{GL}(n)\times\mathrm{GL}(n')$ Rankin--Selberg $L$-functions $L(s,\pi\times\pi')$ over a number field $F$.  In contrast with all earlier work on zero-free regions for these $L$-functions \cite{Banks,Brumley,GelbartLapid,GoldfeldLi,HarcosThorner_Tatuzawa,HarcosThorner,HoffsteinLockhart,HoffsteinRamakrishnan,Humphries2,Humphries,HumphriesThorner,JacquetShalika,Lapid,Luo,Moreno,RamakrishnanWang,Sarnak,Shahidi,Siegel,Tatuzawa,Thorner_Siegel,ThornerZhao,Wattanawanichkul}, our method, which is inspired by proofs of log-free zero density estimates \cite{Bombieri,Gallagher,ST}, relies on two distinct lower bounds for power sums in complementary roles.  Our results are equivalent to certain unconditional ineffective lower bounds:  If $\mathfrak{C}_{\pi}$ is the analytic conductor of $\pi$, $t\in\mathbb{R}$, and $\varepsilon>0$, then
\[
|L(1+it,\pi)|\gg_{n,F,\varepsilon}(\mathfrak{C}_{\pi}(|t|+3))^{-\varepsilon},\qquad |L(1+it,\pi\times\pi')|\gg_{n,F,\pi',\varepsilon}(\mathfrak{C}_{\pi}(|t|+3))^{-\varepsilon}.
\]

\subsection{Standard $L$-functions}
\label{subsec:intro_std}

Throughout, we fix positive integers $n$ and $n'$, and a number field $F$ with ring of adeles $\mathbb{A}_F$.  Let $\mathfrak{F}_n$ be the set of cuspidal automorphic representations $\pi$ of $\mathrm{GL}_n(\mathbb{A}_F)$ with unitary central character.  For $\pi\in\mathfrak{F}_n$ and $s=\sigma+it$, let $L(s,\pi)$ be the standard $L$-function \cite{GodementJacquet} and $\mathfrak{C}_{\pi}$ the analytic conductor.  The generalized Riemann hypothesis (GRH) asserts that $L(\sigma+it,\pi)\neq 0$ when $\sigma>\frac{1}{2}$.  Jacquet~and~Shalika~\cite{JacquetShalika} proved that $L(\sigma+it,\pi)\neq 0$ for $\sigma\geq 1$, extending work of Hadamard and de la Vall{\'e}e Poussin for the Riemann zeta function.  Let $\mathfrak{F}_n^*$ be the set of $\pi\in\mathfrak{F}_n$ such that the central character of $\pi$ is trivial on the diagonally embedded positive reals.  If $|\cdot|$ denotes the idelic norm, then for all $\pi\in\mathfrak{F}_n$, there exists a unique pair $(\pi^*,t_{\pi})\in\mathfrak{F}_n^*\times\mathbb{R}$ such that $\pi=\pi^*\otimes|\cdot|^{it_{\pi}}$ and $L(\sigma,\pi)=L(\sigma+it_{\pi},\pi^*)$.  Therefore, once we replace $\pi$ with $\pi[it]=\pi\otimes|\cdot|^{it}$ and vary $t$, it suffices to study when $L(\sigma,\pi)\neq 0$.

All known wider zero-free regions taper with $\mathfrak{C}_{\pi}$.  Let $\widetilde{\pi}\in\mathfrak{F}_n$ be contragredient to $\pi$, so that $\pi$ is self-dual if and only if $\pi=\widetilde{\pi}$.   If $n\in\{2,3\}$ or $\pi\neq\widetilde{\pi}$, then there exists an effectively computable constant $c_{1}=c_{1}(n,F)>0$ such that
\begin{equation}
\label{eqn:standard_ZFR_cusp}
L(\sigma,\pi)\neq 0,\qquad \sigma\geq 1-c_{1}(\log\mathfrak{C}_{\pi})^{-1}.
\end{equation}
This standard zero-free region matches classical results for Dirichlet $L$-functions \cite[Section~5.9]{IK}.  See \cite{Banks,HoffsteinLockhart,HoffsteinRamakrishnan,Humphries,Wattanawanichkul}.  In the remaining case where $\pi=\widetilde{\pi}$ and $L(s,\pi)$ is not the $L$-function of a real-valued Hecke character over $F$ (so $\pi\in\bigcup_{n=4}^{\infty}\mathfrak{F}_n$), Hoffstein~and~Ramakrishnan~\cite[Theorem~B]{HoffsteinRamakrishnan} proved that if one assumes the Langlands principle of functoriality, in particular the existence of the automorphic tensor product (see \cite[Hypothesis~H$(\pi)$]{HoffsteinRamakrishnan} for a complete statement of their hypothesis), then $L(\sigma,\pi)\neq 0$ in \eqref{eqn:standard_ZFR_cusp}.  For arbitrary self-dual $\pi\in\bigcup_{n=4}^{\infty}\mathfrak{F}_n$, such a hypothesis is a distant dream at the moment.

It follows from the unconditional work of Brumley~\cite{Brumley,Lapid} that if $\pi\in\mathfrak{F}_n$ is arbitrary, then for all $\varepsilon>0$, there exists an effectively computable constant $c_{2}=c_{2}(n,F,\varepsilon)>0$ such that
\begin{equation}
\label{eqn:Brumley_cusp}
|L(\sigma,\pi)|\geq c_{2}\mathfrak{C}_{\pi}^{-2n-\varepsilon},\qquad \sigma\geq 1-c_{2}\mathfrak{C}_{\pi}^{-2n-\varepsilon}.
\end{equation}
Harcos~and~Thorner~\cite{HarcosThorner_Tatuzawa,HarcosThorner} improved \eqref{eqn:Brumley_cusp} in the $\mathrm{GL}_1$-twist aspect:  For all $\pi\in\mathfrak{F}_n$ and $\varepsilon>0$, there exists an ineffective constant $c_{3}=c_{3}(\pi,\varepsilon)>0$ such that if $\chi\in\mathfrak{F}_1$, then
\begin{equation}
\label{eqn:HarcosThorner_cusp}
|L(\sigma,\pi\otimes\chi)|\geq c_{3}\mathfrak{C}_{\chi}^{-\varepsilon},\qquad \sigma\geq 1-c_{3}\mathfrak{C}_{\chi}^{-\varepsilon}.
\end{equation}
Also, there exists $\chi_{\pi,\varepsilon}\in\mathfrak{F}_1$ such that if $L(s,\pi\otimes\chi)$ differs from $L(s,\pi\otimes\chi_{\pi,\varepsilon})$, then \eqref{eqn:HarcosThorner_cusp} holds for $L(s,\pi\otimes\chi)$ with $c_{3}$ effective and $\gg_{n,F,\varepsilon}\mathfrak{C}_{\pi}^{-\varepsilon}$.  When $F=\mathbb{Q}$ and $\pi=\mathbbm{1}$ (the trivial representation in $\mathfrak{F}_1$), the work in \cite{HarcosThorner_Tatuzawa,HarcosThorner} recovers work of Siegel and Tatuzawa \cite{Siegel,Tatuzawa}.

Here, we introduce a new approach to zero-free regions that significantly improves upon \eqref{eqn:Brumley_cusp} uniformly with respect to $\mathfrak{C}_{\pi}$, in a manner that subsumes \eqref{eqn:HarcosThorner_cusp}.  Our main result, Theorem~\ref{thm:main}, implies the following theorem.
\begin{theorem}
\label{thm:standard}
For all $\varepsilon>0$, there exists an ineffective constant $c_{4}=c_{4}(n,F,\varepsilon)>0$ such that if $\pi\in\mathfrak{F}_n$ and $\sigma\geq 1-c_{4}\mathfrak{C}_{\pi}^{-\varepsilon}
$, then $|L(\sigma,\pi)|\geq c_{4}\mathfrak{C}_{\pi}^{-\varepsilon}$.
\end{theorem}
\begin{remark}
By replacing $\pi$ with $\pi[it]$ and using \eqref{eqn:BH}, we find that for all $\varepsilon>0$, there exists an ineffective constant $c_{5}=c_{5}(n,F,\varepsilon)>0$ such that if $\sigma\geq 1-c_{5}\,(\mathfrak{C}_{\pi}(|t|+3))^{-\varepsilon}$, then $|L(\sigma+it,\pi)|\geq c_{5}\,(\mathfrak{C}_{\pi}(|t|+3))^{-\varepsilon}$.
\end{remark}

\subsection{Rankin--Selberg $L$-functions}
\label{subsec:intro_RS}

Theorem~\ref{thm:standard} follows from a more general result.  Given $(\pi,\pi')\in\mathfrak{F}_n\times\mathfrak{F}_{n'}$, let $L(s,\pi\times\pi')$ be the associated Rankin--Selberg $L$-function \cite{JPSS,JS1,JS2}.  Langlands has conjectured that $L(s,\pi\times\pi')$ is a product of standard $L$-functions (i.e., that $L(s,\pi\times\pi')$ is modular), but this is known only in special cases.

GRH asserts that $L(\sigma+it,\pi\times\pi')\neq 0$ when $\sigma>\frac{1}{2}$.  Shahidi~\cite{Shahidi} proved that if $\sigma\geq 1$, then $L(\sigma+it,\pi\times\pi')\neq 0$.  Replacing $\pi$ with $\pi[it]$, it suffices to discuss the nonvanishing of $L(\sigma,\pi\times\pi')$ for $\sigma\in\mathbb{R}$.  To describe refinements to Shahidi's work, we note that if $\pi=\pi^*[it_{\pi}]$ and $\pi'=\pi'^*[it_{\pi'}]$, then $L(s,\pi\times\pi')=L(s+i(t_{\pi}+t_{\pi'}),\pi^*\times\pi'^*)$.  Now, $L(s,\pi\times\pi')$ is entire if and only if $\widetilde{\pi}^*\neq\pi'^*$.  If $\pi^*=\widetilde{\pi}'^*$, then $L(s,\pi\times\pi')$ is holomorphic away from a simple pole at $s=1-i(t_{\pi}+t_{\pi'})$.  It is proved in \cite{Humphries,HumphriesThorner,Moreno,Wattanawanichkul} that there exists an effectively computable constant $c_{6}=c_{6}(n,n',F)>0$ such that if
\begin{equation}
\label{eqn:hypotheses_RS_std_ZFR}
\pi^*=\widetilde{\pi}^*\qquad\textup{or}\qquad \pi'^*=\widetilde{\pi}'^*\qquad\textup{or}\qquad L(s,\pi^*\times\pi'^*)=L(s,\widetilde{\pi}^*\times\widetilde{\pi}'^*),
\end{equation}
then $L(\sigma+it,\pi^*\times\pi'^*)$ has at most one exceptional zero in the region
\begin{equation}
\label{eqn:std_RS}
\sigma\geq 1-c_{6}/\log(\mathfrak{C}_{\pi^*}\mathfrak{C}_{\pi'^*}(|t|+3)).
\end{equation}
If the exception exists, then it is real and simple, and $L(s,\pi^*\times\pi'^*)=L(s,\widetilde{\pi}^*\times\widetilde{\pi}'^*)$.  A natural generalization of \cite[Hypothesis~H$(\pi)$]{HoffsteinRamakrishnan} will imply that if certain Rankin--Selberg $L$-functions depending on $\pi$ and $\pi'$ are modular, then the zero-free region \eqref{eqn:std_RS} holds for all $\mathrm{GL}_n\times\mathrm{GL}_{n'}$ Rankin--Selberg $L$-functions unless they are divisible by the $L$-function of a self-dual element of $\mathfrak{F}_1$.

Brumley~\cite{Brumley,Lapid} (cf.\,\cite{Zhang}) proved that for all $\varepsilon>0$, there exists an effectively computable constant $c_{7}=c_{7}(n,n',F,\varepsilon)>0$ such that for all $\pi\in\mathfrak{F}_n$ and $\pi'\in\mathfrak{F}_{n'}$, we have
\begin{equation}
\label{eqn:Brumley_RS}
|L(\sigma,\pi\times\pi')|\geq c_{7}(\mathfrak{C}_{\pi}\mathfrak{C}_{\pi'})^{-2(n+n'-1)-\varepsilon},\qquad\sigma\geq 1-c_{7}(\mathfrak{C}_{\pi}\mathfrak{C}_{\pi'})^{-2(n+n'-1)-\varepsilon}.
\end{equation}
Harcos and Thorner \cite{HarcosThorner,HarcosThorner_Tatuzawa} proved that for all $\pi\in\mathfrak{F}_n$, $\pi'\in\mathfrak{F}_{n'}$, and $\varepsilon>0$, there exists an ineffective constant $c_{8}=c_{8}(\pi,\pi',\varepsilon)>0$ such that if $\chi\in\mathfrak{F}_1$, then
\begin{equation}
\label{eqn:HarcosThorner_RS}
|L(\sigma,\pi\times(\pi'\otimes\chi))|\geq c_{8}\mathfrak{C}_{\chi}^{-\varepsilon},\qquad \sigma\geq 1-c_{8}\mathfrak{C}_{\chi}^{-\varepsilon}.
\end{equation}
Also, there exists $\chi_{\pi,\pi',\varepsilon}\in\mathfrak{F}_1$ such that if $L(s,\pi\times(\pi'\otimes\chi))$ differs from $L(s,\pi\times(\pi'\otimes\chi_{\pi,\pi',\varepsilon}))$, then \eqref{eqn:HarcosThorner_RS} holds for $L(s,\pi\times(\pi'\otimes\chi))$ with $c_{8}$ effective and of size $\gg_{n,n',F,\varepsilon}(\mathfrak{C}_{\pi}\mathfrak{C}_{\pi'})^{-\varepsilon}$.

Our main result is a substantial uniform improvement over \eqref{eqn:Brumley_RS} when $\pi'$ is fixed, in a manner subsuming \eqref{eqn:HarcosThorner_RS}.  We recover Theorem~\ref{thm:standard} when $\pi'=\mathbbm{1}$, since $L(s,\pi\times\mathbbm{1})=L(s,\pi)$.

\begin{theorem}
\label{thm:main}
Fix $\pi'\in\mathfrak{F}_{n'}$.  For all $\varepsilon>0$, there exists an ineffective constant $c_{9}=c_{9}(n,F,\pi',\varepsilon)>0$ such that if $\pi\in\mathfrak{F}_n$ and $\sigma\geq 1-c_{9}\mathfrak{C}_{\pi}^{-\varepsilon}$, then $|L(\sigma,\pi\times\pi')|\geq c_{9}\mathfrak{C}_{\pi}^{-\varepsilon}$.
\end{theorem}

\begin{remark}
The constant $c_{9}$ is independent of the varying representation $\pi$, but dependence on the fixed $\pi'$ is allowed.  By replacing $\pi$ with $\pi[it]$ and using \eqref{eqn:BH}, we find that for all $\varepsilon>0$, there exists an ineffective constant $c_{10}=c_{10}(n,F,\pi',\varepsilon)>0$ such that
\begin{equation}
\label{eqn:t-shifted}
|L(\sigma+it,\pi\times\pi')|\geq c_{10}(\mathfrak{C}_{\pi}(|t|+3))^{-\varepsilon},\qquad \sigma\geq 1-c_{10}(\mathfrak{C}_{\pi}(|t|+3))^{-\varepsilon}.
\end{equation}
\end{remark}

\subsection{The need for a different approach}

The approach to a standard zero-free region for an $L$-function $L(s)$ in \cite{Banks,HoffsteinLockhart,HoffsteinRamakrishnan,Humphries,HumphriesThorner,Lapid,Moreno,Wattanawanichkul} crucially relies on constructing an auxiliary product $D(s)$ of shifted $L$-functions such that $D(s)$ has nonnegative Dirichlet coefficients and the property that if $1+it_0$ is a pole of $D(s)$ of order $k\geq 1$, then there exists $\ell>k$ such that in a neighborhood of $1+it_0$, $D(s)$ is well-approximated by $|L(1)|^{\ell}$.  This mirrors the strategy for the standard zero-free region for Dirichlet $L$-functions.  Hoffstein~and~Ramakrishnan~\cite{HoffsteinRamakrishnan} proved that if certain Rankin--Selberg $L$-functions are modular (as conjectured by Langlands, see their Hypothesis H$(\pi)$), then one can always construct such a $D(s)$, resulting in the zero-free regions \eqref{eqn:standard_ZFR_cusp} and \eqref{eqn:std_RS}, except for $L(s,\pi)$ when $\pi\in\mathfrak{F}_1$ and $\pi=\widetilde{\pi}$.  There are limited cases of self-dual $\pi\in\bigcup_{n=2}^{\infty}\mathfrak{F}_n$, or of Rankin--Selberg $L$-functions $L(s,\pi\times\pi')$ not satisfying \eqref{eqn:hypotheses_RS_std_ZFR}, for which existing progress towards the modularity of Rankin--Selberg $L$-functions suffices to construct a suitable $D(s)$ \cite{Luo,RamakrishnanWang,Thorner_Siegel,ThornerZhao}.  The weaker zero-free regions in  \cite{Brumley,HarcosThorner,HarcosThorner_Tatuzawa,Siegel,Tatuzawa} also rely on constructing a suitable $D(s)$ with nonnegative Dirichlet coefficients, but the inequality $\ell>k$ is replaced by $\ell\geq k$.  One can also use the Maa{\ss}--Selberg relations \cite{GelbartLapid,GoldfeldLi,Humphries2,Sarnak}, but the resulting $\mathfrak{C}_{\pi}$-uniformity is weaker than \eqref{eqn:Brumley_cusp} or \eqref{eqn:Brumley_RS}.

Is a new approach to zero-free regions genuinely necessary to prove Theorems~\ref{thm:standard}~and~\ref{thm:main}?  Indeed, one might seek a direct generalization of Goldfeld's proof \cite{Goldfeld_siegel} of Siegel's ineffective lower bound $L(1,\chi)\gg_{\varepsilon}q_{\chi}^{-\varepsilon}$ \cite{Siegel}, much like the results of \cite{HarcosThorner_Tatuzawa,HarcosThorner} in the special case of character twists.  It turns out that for arbitrary self-dual $\pi\in\bigcup_{n=4}^{\infty}\mathfrak{F}_n$, the approaches described above (based on nonnegativity of Dirichlet coefficients of an auxiliary Dirichlet series) cannot produce a proof of Theorem~\ref{thm:standard} without the Hoffstein--Ramakrishnan Hypothesis~H$(\pi)$.  A similar statement is true for $L(s,\pi\times\pi')$ failing to satisfy \eqref{eqn:hypotheses_RS_std_ZFR} (aside from the exceptions already mentioned in Section~\ref{subsec:intro_RS}).  Proofs of careful formulations of these informally stated claims will appear in forthcoming work by Nawapan Wattanawanichkul.


\subsection*{Notation}
\label{subsec:notation}
We write $\mathbb{N}$ for the set of positive integers, and $\varepsilon>0$ denotes a fixed, arbitrarily small quantity.  Throughout, $F$ denotes a number field, and $n,n'\in\mathbb{N}$.  By $f \ll_{\nu} g$ and $f = O_{\nu}(g)$, we mean that in a range of $z\in\mathbb{C}$ that is clear from context, there exists a constant $c=c(\nu,n,n',F)>0$ such that $|f(z)|\leq c|g(z)|$.  Unless otherwise specified, $c$ is effectively computable.  If $\nu$ is absent, then $c$ depends only on $n$, $n'$, and $F$.  The constants in the sequence $c_1,c_2,\ldots$ are positive; unless we specify otherwise, they are effectively computable and depend only on $n$, $n'$, and $F$.
\subsection*{Acknowledgements}

The author thanks Dorian Goldfeld, Gergely Harcos, Yujiao Jiang, and Asif Zaman for helpful conversations.  The author is partially supported by the Simons Foundation (MP-TSM-00002484) and the National Science Foundation (DMS-2401311).

\subsection*{Statement of AI usage}

No mathematical argument generated by a large language model (LLM) is used in this preprint.  After this preprint was written, but before it was posted on arXiv, the author prompted two LLMs (Claude Fable 5 on ``Max'' effort, ChatGPT 5.6 Sol on ``Pro'' effort) to ascertain whether they could prove Theorem 1.1.  Despite sharing references in this paper for context and suggesting that it explore the tools described in Section 2 (power sums and zero repulsion), both LLMs were unable to prove Theorem 1.1 unconditionally using only the current knowledge on modularity of Rankin--Selberg $L$-functions.  Faithful, time-stamped transcripts of these conversations are available upon request.  After these conversations, the LLMs assisted with proofreading and error detection.  All suggestions were independently checked by the author.

\section{Strategy for Theorem~\ref{thm:main}}
\label{sec:new_tool}

\subsection{Lower bounds for power sums}

Our proof of Theorem~\ref{thm:main} relies on two lower bounds for power sums, both relying on the ideas of S{\'o}s~and~Tur{\'a}n~\cite{SosTuran,TuranBook}.  One widely used modern reference is \cite{Montgomery_Ten}.  The first lower bound we cite is a conveniently weakened form of Tur{\'a}n's Second Main Theorem.

\begin{proposition}[{\cite[Theorem~4,~p.~94]{Montgomery_Ten}}]
\label{prop:Turan}
Let $z_1,\ldots,z_{\nu}\in\mathbb{C}$.  If $K\geq \nu$, then there exists an integer $\ell\in[K,2K]$ such that $|z_1^{\ell}+\cdots+z_{\nu}^{\ell}|\geq (|z_1|/50)^{\ell}$.
\end{proposition}

The second lower bound we need is for real parts of power sums, a very different result.

\begin{proposition}[{\cite[Lemma~2,~p.~173]{Montgomery_Ten}}]
\label{prop:Turan2}
Let $m\geq 1$ be an integer, and let $(b_n)_{n=1}^{\infty}$ be a sequence of real numbers and $(z_n)_{n=1}^{\infty}$ a sequence of complex numbers.  Define $s_m = \sum_{n=1}^{\infty}b_n z_n^m$.  Suppose that $\sup_n|z_n|=|z_1|$, $b_n\geq 0$ for all $n$ such that $|z_n|\geq \frac{1}{3}|z_1|$, and $b_1|z_1|\neq 0$.  Define $L=(b_1 |z_1|)^{-1}\sum_{n=1}^{\infty}|b_n z_n|$.  There exists an integer $j\in[1,24L]$ such that $\mathrm{Re}(s_{j})\geq b_1|z_1|^j/8$.
\end{proposition}

\subsection{The strategy}

\subsubsection*{Step 1 (Sections~\ref{subsec:isobaric} and \ref{sec:start})}

Let $\pi'\in\mathfrak{F}_{n'}$, and fix $\varepsilon\in(0,(100n'n)^{-2})$.  If each $\pi\in\mathfrak{F}_n$ satisfies $L(\sigma,\pi\times\pi')\neq 0$ for $\sigma\geq 1-\varepsilon$, then we are done.  Otherwise, there exists $\pi_{\varepsilon}\in\mathfrak{F}_n$ (depending at most on $n$, $F$, $\pi'$, and $\varepsilon$) such that $L(s,\pi_{\varepsilon}\times\pi')$ has a real zero $\beta_{\varepsilon}$ in the interval $(1-\varepsilon,1)$.  Using \eqref{eqn:Brumley_RS} and \eqref{eqn:HarcosThorner_RS}, we reduce to the case where $L(s,\pi\times\pi')$, $L(s,\pi\times\widetilde{\pi}')$, $L(s,\pi_{\varepsilon}\times\pi')$, $L(s,\pi_{\varepsilon}\times\widetilde{\pi}')$, $L(s,\pi\times\pi_{\varepsilon})$, and $L(s,\pi\times\widetilde{\pi}_{\varepsilon})$ are all entire and $\mathfrak{C}_{\pi}>\max\{\mathfrak{C}_{\pi_{\varepsilon}},\mathfrak{C}_{\pi'}\}$.  In this setting, we can generalize classical results of Landau and Page on the rarity of real zeros of Dirichlet $L$-functions close to $s=1$ (Proposition~\ref{prop:LandauPage}).  This permits us to further reduce to the case where $L(s,\pi\times\pi')$ has a real zero $\beta_1$ satisfying
\begin{equation}
\label{eqn:step1bound}
1-\varepsilon<\beta_{\varepsilon}<1-\frac{1}{126(2n+n')\log\mathfrak{C}_{\pi}}<\beta_1 < 1.
\end{equation}

\subsubsection*{Step 2 (Sections \ref{sec:lower_G}-\ref{sec:upper_G})}

Let $K\geq 1$, and let $k\in[K,2K+1]$ be an integer.  With $\pi_{\varepsilon}$, $\beta_{\varepsilon}$, and $\beta_1$ as in Step 1, we consider
\[
\mathscr{F}(z)=L(z,\pi_{\varepsilon}\times\pi')L(z,\pi_{\varepsilon}\times\widetilde{\pi}),\qquad G_k(z)=\frac{(-1)^{k}}{k!}\Big(\frac{\mathscr{F}'}{\mathscr{F}}\Big)^{(k)}(z).
\]
Proposition~\ref{prop:LandauPage} implies that $\beta_{\varepsilon}$ is a singularity of $G_k(z)$, but not $\beta_1$.  By considering the Hadamard factorization of $\mathscr{F}(z)$ in our expression for $G_k(z)$, we can use Proposition~\ref{prop:Turan} to prove that there exists a constant $c_{21}\geq 2$ such that if $K\geq c_{21}\varepsilon \log\mathfrak{C}_{\pi}$, then there exists $k\in [K,2K+1]$ such that
\begin{equation}
\label{eqn:G_lower}
\frac{1}{2(100\varepsilon)^{k+1}}\leq |G_k(1+\varepsilon)|.
\end{equation}
Since $\beta_1$ is not a singularity of $G_k(z)$, Proposition~\ref{prop:Turan} shows that $|G_k(1+\varepsilon)|$ is large because of  $\beta_{\varepsilon}$ (a method of zero detection).

We bound $|\varepsilon^{k+1}G_k(1+\varepsilon)|$ from above by expressing it as an absolutely convergent Dirichlet series.  If $c_{21}$ is large and $N_{\varepsilon} = \exp(K/(300\varepsilon))$, then $|\varepsilon^{k+1}G_k(1+\varepsilon)|$ is
\begin{equation}
\label{eqn:G_upper}
\leq  \varepsilon^2\int_{N_{\varepsilon}}^{N_{\varepsilon}^{12000}}\Big|\sum_{\mathrm{N}\mathfrak{p}\in[N_{\varepsilon},u]}\frac{\lambda_{\pi_{\varepsilon}}(\mathfrak{p})(\lambda_{\pi'}(\mathfrak{p})+\lambda_{\widetilde{\pi}}(\mathfrak{p}))\log\mathrm{N}\mathfrak{p}}{\mathrm{N}\mathfrak{p}}\Big|\frac{du}{u}+\frac{1}{4(100)^{k+1}}.
\end{equation}

\subsubsection*{Step 3 (Section~\ref{sec:finish_G})}

After some manipulations, \eqref{eqn:G_lower} and \eqref{eqn:G_upper} combine to give us
\begin{align*}
1&\ll_{\varepsilon}100^{4K}K^2\Big(\sum_{\mathrm{N}\mathfrak{p}\in[N_{\varepsilon},N_{\varepsilon}^{12000}]}\frac{|\lambda_{\pi_{\varepsilon}}(\mathfrak{p})|^2\log\mathrm{N}\mathfrak{p}}{\mathrm{N}\mathfrak{p}}\Big)\Big(\sum_{\mathrm{N}\mathfrak{p}\in[N_{\varepsilon},N_{\varepsilon}^{12000}]}\frac{|\lambda_{\pi'}(\mathfrak{p})+\lambda_{\widetilde{\pi}}(\mathfrak{p})|^2\log\mathrm{N}\mathfrak{p}}{\mathrm{N}\mathfrak{p}}\Big).
\end{align*}
A short calculation shows that the first sum over $\mathfrak{p}$ is $\ll_{\varepsilon}\log N_{\varepsilon}\ll_{\varepsilon} K$, hence
\begin{equation}
\label{eqn:upper_sums}
1\ll_{\varepsilon}100^{4K}K^3\sum_{\mathrm{N}\mathfrak{p}\in[N_{\varepsilon},N_{\varepsilon}^{12000}]}\frac{|\lambda_{\pi'}(\mathfrak{p})+\lambda_{\widetilde{\pi}}(\mathfrak{p})|^2\log\mathrm{N}\mathfrak{p}}{\mathrm{N}\mathfrak{p}}.
\end{equation}

\subsubsection*{Step 4 (Section~\ref{sec:partial_sums})}

Using \eqref{eqn:step1bound}, we prove the existence of a constant $c_{22}\geq c_{21}$ such that
\begin{equation}
\label{eqn:step4bound}
\sum_{\mathrm{N}\mathfrak{p}\in[N_{\varepsilon},N_{\varepsilon}^{12000}]}\frac{|\lambda_{\pi'}(\mathfrak{p})+\lambda_{\widetilde{\pi}}(\mathfrak{p})|^2\log\mathrm{N}\mathfrak{p}}{\mathrm{N}\mathfrak{p}}\ll_{\varepsilon} (1-\beta_1)K^3,\qquad K = c_{22}\varepsilon\log\mathfrak{C}_{\pi}.
\end{equation}
See Proposition~\ref{prop:Linnik}.  The proof relates \eqref{eqn:step4bound} to a contour integral of the negative logarithmic derivative of $D(s) = L(s,\pi\times\widetilde{\pi})L(s,\pi'\times\widetilde{\pi}')L(s,\pi\times\pi')L(s,\widetilde{\pi}\times\widetilde{\pi}')$.  Pushing the contour to the left, we need a sparsity of zeros of $D(s)$ near $\mathrm{Re}(s)=1$, uniformly in $\mathfrak{C}_{\pi}$ and $\mathfrak{C}_{\pi'}$.  Because of the last two terms in $D(s)$, we only have Brumley's narrow zero-free region \eqref{eqn:Brumley_RS} (with $|\cdot|^{it}$ twists).  For our work, this is insufficient on its own.  In addition to \eqref{eqn:Brumley_RS}, we use the assumption of \eqref{eqn:step1bound} and Moreno's version of the Deuring--Heilbronn zero repulsion phenomenon \cite[Theorem 4.2]{Moreno} (Lemma~\ref{lem:DH} below), leading to the crucial factor of $1-\beta_1$.  Moreno's proof of zero repulsion for $D(s)$ is based on Proposition~\ref{prop:Turan2}.

For Dirichlet $L$-functions ($F=\mathbb{Q}$, $\pi'=\mathbbm{1}$, and $\pi\in\mathfrak{F}_1$ corresponds with a quadratic character $\chi$), Bombieri~\cite[\S6]{Bombieri} noted that $|1+\chi(p)|^2\ll 1+\chi(p)$.  Since $(1+\chi(p))\log p$ is the $p$-th Dirichlet coefficient of $-\frac{\zeta'}{\zeta}(s)-\frac{L'}{L}(s,\chi)$, Bombieri could prove \eqref{eqn:step4bound} without an external zero repulsion result.  However, the residue calculations in his method are sensitive to the fact that $-\frac{\zeta'}{\zeta}(s)-\frac{L'}{L}(s,\chi)$ has a simple pole at $s=1$.  Since $-\frac{D'}{D}(s)$ has a pole of order $2$ and a more complicated Laurent expansion at $s=1$, Bombieri's approach does not work in our level of generality without knowing Theorem~\ref{thm:main} in advance.

\subsubsection*{Step 5 (Sections~\ref{sec:finish_G} and \ref{sec:Proof})}

Inserting \eqref{eqn:step4bound} into \eqref{eqn:upper_sums} with $K = c_{22}\varepsilon\log\mathfrak{C}_{\pi}$, we find that $1\ll_{\varepsilon}100^{4K}K^6(1-\beta_1)\ll_{\varepsilon}\mathfrak{C}_{\pi}^{19 c_{22}\varepsilon}(1-\beta_1)$, establishing the zero-free interval in Theorem~\ref{thm:main} and finishing the casework in Step 1.  Replacing $\pi$ with $\pi[it]$, we obtain the zero-free region \eqref{eqn:t-shifted}.  The lower bound in Theorem~\ref{thm:main} follows from \eqref{eqn:t-shifted} and contour integration.

\subsubsection*{What is new}

The proof combines two distinct power-sum inequalities in complementary roles.  The initial reductions, using Proposition \ref{prop:LandauPage} and the zero-free regions \eqref{eqn:Brumley_RS} and \eqref{eqn:HarcosThorner_RS}, reduce the problem to two real zeros $\beta_{\varepsilon}<\beta_1$.  By Proposition~\ref{prop:Turan} applied to the zeros of $\mathscr{F}(z)$ (i.e., the singularities of $G_k(z)$), the existence of $\beta_{\varepsilon}$ forces $G_k(1+\varepsilon)$ to be large.  We use Moreno's version of the zero repulsion phenomenon, which uses Proposition~\ref{prop:Turan2}, to bound $G_k(1+\varepsilon)$ by a multiple of $1-\beta_1$, ensuring that $|G_k(1+\varepsilon)|$ cannot be {\it too} large.  Comparing these estimates yields the first unconditional uniform Siegel-type lower bound for all standard $L$-functions $L(s,\pi)$, and more generally for Rankin--Selberg $L$-functions $L(s,\pi\times\pi')$ with $\pi'$ fixed, without recourse to unproven hypotheses on the modularity of Rankin--Selberg $L$-functions.  Unlike corresponding work for classical Dirichlet $L$-functions, Propositions~\ref{prop:Turan}~and~\ref{prop:Turan2} are both essential to this argument.


\section{Applications}
\label{sec:applications}

\subsection{Prime number theorems}

Let $\mathfrak{a}$ be a nonzero ideal of the ring of integers $\mathcal{O}_F$, $\mathfrak{p}$ a nonzero prime ideal of $\mathcal{O}_F$, $\mathrm{N}\mathfrak{a}=|\mathcal{O}_F/\mathfrak{a}|$, $\Lambda_{\pi}(\mathfrak{a})$ (resp.\,$\Lambda_{\pi\times\pi'}(\mathfrak{a})$) the $\mathfrak{a}$-th Dirichlet coefficient of $-(L'/L)(s,\pi)$ (resp.\,$-(L'/L)(s,\pi\times\pi')$), and $D_F$ the norm of the absolute discriminant of $F$.  Let $\pi\in\mathfrak{F}_n$ and $\pi'\in\mathfrak{F}_{n'}$.  For $x\geq 1$, define
\begin{equation}
\label{eqn:err_def}
\begin{gathered}
\mathcal{E}(x;\pi):=\Big|\frac{1}{x}\sum_{\mathrm{N}\mathfrak{a}\leq x}\Lambda_{\pi}(\mathfrak{a})-\mathbf{1}_{\pi^*=\mathbbm{1}}\frac{x^{-it_{\pi}}}{1-it_{\pi}}\Big|,\\
\mathcal{E}(x;\pi,\pi'):=\Big|\frac{1}{x}\sum_{\mathrm{N}\mathfrak{a}\leq x}\Lambda_{\pi\times\pi'}(\mathfrak{a})-\mathbf{1}_{\pi^*=\widetilde{\pi}'^*}\frac{x^{-i(t_{\pi}+t_{\pi'})}}{1-i(t_{\pi}+t_{\pi'})}\Big|.
\end{gathered}
\end{equation}
The prime number theorem for $L(s,\pi)$, which is the statement $\lim_{x\to\infty}\mathcal{E}(x;\pi)=0$, is equivalent to the nonvanishing of $L(\sigma+it,\pi)$ for all $t\in\mathbb{R}$ and $\sigma\geq 1$.  The prime number theorem for $L(s,\pi\times\pi')$, which is the statement $\lim_{x\to\infty}\mathcal{E}(x;\pi,\pi')=0$, is equivalent to the nonvanishing of $L(\sigma+it,\pi\times\pi')$ for all $t\in\mathbb{R}$ and $\sigma\geq 1$.

When $\pi^*\neq\widetilde{\pi}^*$, one can use the standard zero-free region \eqref{eqn:standard_ZFR_cusp} (which has no exceptional zero) to prove that there exist constants $c_{11}$ and $c_{12}$ such that (cf. \cite[Theorem~2.4]{KanekoThorner})
	\begin{equation}
	\label{eqn:PNT_cusp_non-self-dual}
	\mathcal{E}(x;\pi)\leq c_{11} e^{-\sqrt{\log x}},\qquad \log x\geq c_{12}(\log \mathfrak{C}_{\pi})^2.
	\end{equation}
This matches classical results for $L$-functions associated to complex Dirichlet characters \cite[Section~5.9]{IK}. Using \eqref{eqn:standard_ZFR_cusp} and Theorem~\ref{thm:standard}, we can prove a version of \eqref{eqn:PNT_cusp_non-self-dual} when $\pi^*=\widetilde{\pi}^*$, matching classical results for $L$-functions associated to real Dirichlet characters.
\begin{theorem}
\label{thm:PNT_std}
For all $\varepsilon>0$, there exist ineffective constants $c_{13}=c_{13}(n,F,\varepsilon)>0$ and $c_{14}=c_{14}(n,F,\varepsilon)>0$ such that if $\pi=\pi^*[it_{\pi}]\in\mathfrak{F}_n$, $\pi^*=\widetilde{\pi}^*$, $\log x\geq c_{13}\mathfrak{C}_{\pi}^{\varepsilon}$, and $\mathcal{E}(x;\pi)$ is as in \eqref{eqn:err_def}, then $\mathcal{E}(x;\pi)\leq c_{14} e^{-\sqrt{\log x}}$.
\end{theorem}

If exactly one of $\pi^*$ and $\pi'^*$ is self-dual, then one can use the standard zero-free region \eqref{eqn:std_RS} (which has no exceptional zero) to prove that there exist constants $c_{15}$ and $c_{16}$  such that (cf{.} \cite[Theorem~2.5]{KanekoThorner})
\begin{equation}
\label{eqn:PNT_RS_non-self-dual}
\mathcal{E}(x;\pi,\pi')\leq c_{15}e^{-\sqrt{\log x}},\qquad \log x\geq c_{16}(\log\mathfrak{C}_{\pi}\mathfrak{C}_{\pi'})^2.
\end{equation}
Like \eqref{eqn:PNT_cusp_non-self-dual}, \eqref{eqn:PNT_RS_non-self-dual} matches classical results for $L$-functions of complex Dirichlet characters.  Using Theorem~\ref{thm:main} and the instances of standard zero-free regions recorded in Section~\ref{subsec:intro_RS}, we can obtain new prime number theorems in the remaining cases (cf.\,\cite[Section~2]{KanekoThorner}).

\begin{theorem}
\label{thm:PNT_RS}
For all $\varepsilon>0$ and $\pi'=\pi'^*[it_{\pi'}]\in\mathfrak{F}_{n'}$, there exist ineffective constants $c_{17}$, $c_{18}$, $c_{19}$, and $c_{20}$ (depending on $n$, $F$, $\pi'$, and $\varepsilon$) such that if $\pi=\pi^*[it_{\pi}]\in\mathfrak{F}_n$ and $\mathcal{E}(x;\pi,\pi')$ is as in \eqref{eqn:err_def}, then the following results are true.
\begin{enumerate}
\item If $L(s,\pi^*\times\pi'^*)=L(s,\widetilde{\pi}^*\times\widetilde{\pi}'^*)$, then
\[
\mathcal{E}(x;\pi,\pi')\leq c_{17}e^{-\sqrt{\log x}},\qquad \log x\geq c_{18}\mathfrak{C}_{\pi}^{\varepsilon}.
\]
\item If $\pi^*\neq\widetilde{\pi}^*$, $\pi'^*\neq \widetilde{\pi}'^*$, and $L(s,\pi^*\times\pi'^*)\neq L(s,\widetilde{\pi}^*\times\widetilde{\pi}'^*)$, then
\[
\mathcal{E}(x;\pi,\pi')\leq c_{19}(\log x)^{-1/\varepsilon},\qquad \log x\geq c_{20}\mathfrak{C}_{\pi}^{\varepsilon}.
\]
\end{enumerate}
\end{theorem}

\begin{remark}
Let $\pi\in\mathfrak{F}_n$ and $\pi'\in\mathfrak{F}_{n'}$. Let $\lambda_{\pi}(\mathfrak{a})$ and $\lambda_{\pi\times\pi'}(\mathfrak{a})$ be the $\mathfrak{a}$-th Dirichlet coefficient of $L(s,\pi)$  and $L(s,\pi\times\pi')$, respectively.	Using Jiang's work (\cite{JiangH}, see Lemma~\ref{lem:HypH} below), Theorems~\ref{thm:PNT_std} and \ref{thm:PNT_RS} hold with $\mathcal{E}(x;\pi)$ and $\mathcal{E}(x;\pi,\pi')$ replaced, respectively, by
\begin{equation*}
\Big|\frac{1}{x}\sum_{\mathrm{N}\mathfrak{p}\leq x}\lambda_{\pi}(\mathfrak{p})\log\mathrm{N}\mathfrak{p}-\mathbf{1}_{\pi^*=\mathbbm{1}}\frac{x^{-it_{\pi}}}{1-it_{\pi}}\Big|,\quad \Big|\frac{1}{x}\sum_{\mathrm{N}\mathfrak{p}\leq x}\lambda_{\pi}(\mathfrak{p})\lambda_{\pi'}(\mathfrak{p})\log\mathrm{N}\mathfrak{p}-\mathbf{1}_{\pi^*=\widetilde{\pi}'^*}\frac{x^{-i(t_{\pi}+t_{\pi'})}}{1-i(t_{\pi}+t_{\pi'})}\Big|.
\end{equation*}
\end{remark}

The proofs of Theorems~\ref{thm:PNT_std} and \ref{thm:PNT_RS} rely on relating \eqref{eqn:err_def} to the zero-free regions in Theorems~\ref{thm:standard} and \ref{thm:main} via an explicit formula.  Such a relation is well-understood; see \cite[Theorem~2.1]{KanekoThorner} when $F=\mathbb{Q}$ and $(\pi,\pi')\in\mathfrak{F}_n^*\times\mathfrak{F}_{n'}^*$.  The proof ideas in \cite{KanekoThorner} still apply after we replace rational primes by prime ideals and use \eqref{eqn:BH}, and the remaining contour integration argument is unchanged.  See also \cite{HarcosThorner,IK}.  We omit the details.

\subsection{Brauer--Siegel type theorems}

Let $n\geq 1$, and let $(K_j)_{j=1}^{\infty}$ be a sequence of number fields such that $[K_j:\mathbb{Q}]=n$ and $\lim_{j\to\infty}D_{K_j}=\infty$.  Let $r_1(j)$ (resp. $r_2(j)$) be the number of real (resp. complex) embeddings $K_j\hookrightarrow\mathbb{C}$.  If $h_j$, $R_j$, and $w_j$ are the class number, regulator, and number of roots of unity in $K_j$, respectively, then
\begin{equation}
\label{eqn:class_number_formula}
\zeta_{K_j}^*(1)=\mathop{\mathrm{Res}}_{s=1}\zeta_{K_j}(s)
=\frac{2^{r_1(j)}(2\pi)^{r_2(j)}h_j R_j}{w_jD_{K_j}^{1/2}}.
\end{equation}
This is the analytic class number formula.  Building on Siegel~\cite{Siegel}, Brauer~\cite{Brauer} proved that
\begin{equation}
\label{eqn:BS}
\lim_{j\to\infty}\frac{\log \zeta_{K_j}^*(1)}{\log D_{K_j}}=0,\qquad\textup{hence}\qquad \lim_{j\to\infty}\frac{\log(h_j R_j)}{\log D_{K_j}}=\frac{1}{2}.
\end{equation}
Thus, $h_j R_j$ enjoys a strong asymptotic lower bound.  We generalize \eqref{eqn:BS} as follows.
\begin{theorem}
\label{thm:Brauer-Siegel}
Fix $n,n'\geq 1$, a number field $F$, and $\pi'\in\mathfrak{F}_{n'}$.  If $(\pi_j)_{j=1}^{\infty}$ is a sequence in $\mathfrak{F}_n$ such that $\lim_{j\to\infty}\mathfrak{C}_{\pi_j}=\infty$, then
\[
\lim_{j\to\infty}\frac{\log|L(1,\pi_j)|}{\log\mathfrak{C}_{\pi_j}}=0,\qquad \lim_{j\to\infty}\frac{\log|L(1,\pi_j\times\pi')|}{\log\mathfrak{C}_{\pi_j}}=0.
\]
\end{theorem}
\begin{proof}
This follows immediately from the combination of the lower bounds in Theorems~\ref{thm:standard} and \ref{thm:main} with the upper bounds in Lemma~\ref{lem:Li1}.
\end{proof}

We briefly describe a conjectural application of Theorem~\ref{thm:Brauer-Siegel}.  To a smooth projective variety $X/\mathbb{Q}$, an integer $i\geq 0$, and $n\in\mathbb{Z}$, one hopes to associate a motive $M=h^{i}(X)(n)$ (i.e., a universal cohomology group of $X$) of pure weight $w=i-2n$.  Let $\mathscr{L}(s,M)$ be the Hasse--Weil $L$-function of $M$, which converges absolutely for $\mathrm{Re}(s)>1+\frac{w}{2}$, and $\mathscr{L}(s,M_{\infty})$ the associated product of gamma factors.  Conjecturally, $\mathscr{L}(s,M)\mathscr{L}(s,M_{\infty})$ meromorphically continues to $\mathbb{C}$, apart from finitely many poles on $\mathrm{Re}(s)=1+\frac{w}{2}$, and satisfies a functional equation that relates $\mathscr{L}(s,M)\mathscr{L}(s,M_{\infty})$ to $\mathscr{L}(w+1-s,M)\mathscr{L}(w+1-s,M_{\infty})$.  See \cite{BlochKato,MR1265544} for known and conjectured properties of motives and their Hasse--Weil $L$-functions.  Assume that neither $\mathscr{L}(s,M_{\infty})$ nor $\mathscr{L}(w+1-s,M_{\infty})$ has a pole at $s=0$, so that $s=0$ is a critical point in the sense of Deligne.  The Bloch--Kato conjecture \cite[Conjecture~5.15~and~(5.15.1)]{BlochKato} relates $\mathscr{L}(0,M)$ to the order of the (conjecturally finite) Shafarevich--Tate group of $M$.  This generalizes conjectures of Be\u{\i}linson, Birch and Swinnerton-Dyer, and Deligne.

Fix $d\geq 1$.  Consider a sequence of such motives $(M_j)_{j=1}^{\infty}$ of weight $w=-2$ and  dimension $d$.  Let $q_{M_j}$ be the global conductor of $M_j$, and assume that $\lim_{j\to\infty}q_{M_j}=\infty$.  Assume that each $M_j$ is modular, so that $L(s,M_j):=\mathscr{L}(s+\frac{w}{2},M_j)$ is a product of $L$-functions of cuspidal automorphic representations.  If, additionally, each $M_j$ satisfies the Bloch--Kato conjecture, then since $\mathscr{L}(0,M_j)=L(1,M_j)$, Theorem~\ref{thm:Brauer-Siegel} implies a strong upper and lower bound for the right-hand side of the Bloch--Kato conjecture (which includes the order of the Shafarevich--Tate group of $M_j$ as a factor) as $q_{M_j}\to\infty$, much like \eqref{eqn:BS} does for $h_j R_j$.

\section{Properties of \texorpdfstring{$L$}{L}-functions}
\label{sec:L-functions}

We recall some facts about automorphic  $L$-functions \cite{GodementJacquet} and Rankin--Selberg $L$-functions \cite{JPSS,JS1,JS2,Shahidi}.  Given $\pi\in\mathfrak{F}_n$, let $\mathfrak{q}_{\pi}$ be the conductor of $\pi$, a nonzero ideal of $\mathcal{O}_F$.  We express $\pi$ as a restricted tensor product $\bigotimes_v \pi_v$ of smooth admissible representations of $\mathrm{GL}_n(F_v)$, where $v$ varies over places of $F$.  If $v$ is archimedean, then we write $v\mid\infty$, and $v\nmid\infty$ otherwise.  If $v\nmid\infty$ corresponds with a prime ideal $\mathfrak{p}$, then we write $\pi_{\mathfrak{p}}$ instead of $\pi_v$.  The $\mathfrak{p}\mid\mathfrak{q}_{\pi}$ are the prime ideals for which $\pi_{\mathfrak{p}}$ is ramified.  Let $\mathfrak{F}_n^*$ be the set of $\pi\in\mathfrak{F}_n$ whose central character is trivial on the diagonally embedded positive reals and unitary.  For all $\pi\in\mathfrak{F}_n$, there exists a unique pair $(\pi^*,t_{\pi})\in\mathfrak{F}_n^*\times\mathbb{R}$ such that $\pi=\pi^*[it_{\pi}]:=\pi^*\otimes|\cdot|^{it_{\pi}}$.

\subsection{Standard \texorpdfstring{$L$}{L}-functions}

Let $\pi\in\mathfrak{F}_n^*$.  For each place $v\nmid\infty$ of $F$ with corresponding prime ideal $\mathfrak{p}$, the local $L$-function $L(s,\pi_v)=L(s,\pi_{\mathfrak{p}})$ is defined in terms of the Langlands parameters $\alpha_{1,\pi}(\mathfrak{p}),\ldots,\alpha_{n,\pi}(\mathfrak{p})$ by
\begin{equation}
	\label{eqn:Euler_p_single}
L(s,\pi_{\mathfrak{p}})=\prod_{j=1}^{n}\frac{1}{1-\alpha_{j,\pi}(\mathfrak{p})\mathrm{N}\mathfrak{p}^{-s}}=1+\sum_{j=1}^{\infty}\frac{\lambda_{\pi}(\mathfrak{p}^j)}{\mathrm{N}\mathfrak{p}^{js}}.
\end{equation}
If $\mathfrak{p}\nmid\mathfrak{q}_{\pi}$, then $\alpha_{j,\pi}(\mathfrak{p})\neq0$ for all $j$.  Otherwise, there exists $j$ such that $\alpha_{j,\pi}(\mathfrak{p})=0$.  The standard $L$-function $L(s,\pi)$ associated to $\pi$ is
\[
L(s,\pi)=\prod_{\mathfrak{p}} L(s,\pi_{\mathfrak{p}})=\sum_{\mathfrak{a}}\frac{\lambda_{\pi}(\mathfrak{a})}{\mathrm{N}\mathfrak{a}^s}.
\]
This Euler product converges absolutely for $\mathrm{Re}(s)>1$.

At each $v\mid\infty$, there are $n$ Langlands parameters $\mu_{j,\pi}(v)\in\mathbb{C}$ such that
\[
L(s,\pi_v) = \prod_{j=1}^{n}\Gamma_{v}(s+\mu_{j,\pi}(v)),\qquad \Gamma_{v}(s)= \begin{cases}
	\pi^{-s/2}\Gamma(s/2)&\mbox{if $F_{v}=\mathbb{R}$,}\\
	2(2\pi)^{-s}\Gamma(s)&\mbox{if $F_{v}=\mathbb{C}$.}
\end{cases}
\]
We then define $L(s,\pi_{\infty})=\prod_{v\mid\infty}L(s,\pi_v)$.  By \cite{LRS,MS}, we know that there exists $\theta_n \in[0,\frac{1}{2}-\frac{1}{n^2+1}]$  such that $|\alpha_{j,\pi}(\mathfrak{p})|\leq \mathrm{N}\mathfrak{p}^{\theta_n}$ and $\mathrm{Re}(\mu_{j,\pi}(v))\geq -\theta_n$.

Let $\widetilde{\pi}\in\mathfrak{F}_n^*$ be the contragredient representation.  We have the equalities $\{\alpha_{j,\widetilde{\pi}}(\mathfrak{p})\}=\{\overline{ \alpha_{j,\pi}(\mathfrak{p})}\}$ and $\{\mu_{j,\widetilde{\pi}}(v)\}=\{\overline{\mu_{j,\pi}(v)}\}$, and $\mathfrak{q}_{\pi}=\mathfrak{q}_{\widetilde{\pi}}$.

Recall that there exists a unique pair $(\pi^*,t_{\pi})\in\mathfrak{F}_n^*\times\mathbb{R}$ such that $\pi=\pi^*[it_{\pi}]$ and $L(s,\pi)=L(s+it_{\pi},\pi^*)$.  We define $r_{\pi}=\mathbf{1}_{\pi^*=\mathbbm{1}}$.  If $\pi\in\mathfrak{F}_n^*$ (so that $\pi=\pi^*$), then the completed $L$-function $\Lambda(s,\pi) = (s(1-s))^{r_{\pi}}(D_F^n \mathrm{N}\mathfrak{q}_{\pi})^{s/2}L(s,\pi)L(s,\pi_{\infty})$ is entire of order 1.  There exists a complex number $W(\pi)$ of modulus 1 such that if $s\in\mathbb{C}$, then the functional equation $\Lambda(s,\pi)=W(\pi)\Lambda(1-s,\widetilde{\pi})$ holds.  To handle $\pi\in\mathfrak{F}_n$, we observe that
\[
L(s,\pi_v)=L(s+it_{\pi},\pi_v^*),\qquad L(s,\pi)=L(s+it_{\pi},\pi^*),\qquad \Lambda(s,\pi)=\Lambda(s+it_{\pi},\pi^*).
\]

Let $\pi\in\mathfrak{F}_n^*$ and $t\in\mathbb{R}$.  The analytic conductor of $\pi$ is
\begin{equation}
\label{eqn:analytic_conductor_def}
\mathfrak{C}_{\pi}(it)= D_F^n \mathrm{N}\mathfrak{q}_{\pi}\prod_{v|\infty}\prod_{j=1}^n(3+|it+\mu_{j,\pi}(v)|^{[F_v:\mathbb{R}]}),\qquad \mathfrak{C}_{\pi}= \mathfrak{C}_{\pi}(0).
\end{equation}
If $\pi\in\mathfrak{F}_n$, then $\mathfrak{C}_{\pi}(it) = \mathfrak{C}_{\pi^*}(i(t+t_{\pi}))$ and $\mathfrak{C}_{\pi} = \mathfrak{C}_{\pi^*}(it_{\pi})$.

\begin{lemma}[{\cite[Equation~(2.10)]{HarcosThorner_Tatuzawa}}]
\label{lem:t_pi_bound}
If $\pi\in\mathfrak{F}_n$, then $|t_{\pi}|+3\leq \mathfrak{C}_{\pi}$.
\end{lemma}

Let $\pi\in\mathfrak{F}_n$.  Since $\Lambda(s,\pi)$ is entire of order 1, there exist $A_{\pi},B_{\pi}\in\mathbb{C}$ such that
\[
\Lambda(s,\pi)=e^{A_{\pi}+B_{\pi}s}\prod_{\Lambda(\rho,\pi)=0}\Big(1-\frac{s}{\rho}\Big)e^{s/\rho}.
\]
The zeros $\rho$ are the nontrivial zeros of $L(s,\pi)$.  The poles of $(s+it_{\pi})^{r_{\pi}}L(s,\pi_{\infty})$ are the trivial zeros of $L(s,\pi)$.

\subsection{Rankin--Selberg \texorpdfstring{$L$}{L}-functions}
\label{subsec:RS}

Let $(\pi,\pi')\in\mathfrak{F}_n^*\times\mathfrak{F}_{n'}^*$.  For each $v\nmid\infty$ with corresponding prime ideal $\mathfrak{p}$, Jacquet, Piatetski-Shapiro, and Shalika \cite{JPSS} associate to $\pi_{\mathfrak{p}}$ and $\pi_{\mathfrak{p}}'$ a local Rankin--Selberg $L$-function
\begin{equation}
\label{eqn:RS_Dirichlet_series}
L(s,\pi_{\mathfrak{p}}\times\pi_{\mathfrak{p}}')=\prod_{j=1}^{n}\prod_{j'=1}^{n'}\frac{1}{1-\alpha_{j,j',\pi\times\pi'}(\mathfrak{p}) \mathrm{N}\mathfrak{p}^{-s}}=1+\sum_{j=1}^{\infty}\frac{\lambda_{\pi\times\pi'}(\mathfrak{p}^j)}{\mathrm{N}\mathfrak{p}^{js}}
\end{equation}
and a local conductor $\mathfrak{q}_{\pi_{\mathfrak{p}}\times\pi_{\mathfrak{p}}'}$.  If $\mathfrak{p}\nmid\mathfrak{q}_{\pi}\mathfrak{q}_{\pi'}$, then
\begin{equation}
\label{eqn:separate_dirichlet_coeffs}
\mathfrak{q}_{\pi_{\mathfrak{p}}\times\pi_{\mathfrak{p}}'}=\mathcal{O}_F,\qquad\{\alpha_{j,j',\pi\times\pi'}(\mathfrak{p})\}=\{\alpha_{j,\pi}(\mathfrak{p})\alpha_{j',\pi'}(\mathfrak{p})\}.
\end{equation}
The Rankin--Selberg $L$-function $L(s,\pi\times\pi')$ (absolutely convergent for $\mathrm{Re}(s)>1$) and arithmetic conductor $\mathfrak{q}_{\pi\times\pi'}$ associated to $\pi$ and $\pi'$ are
\[
L(s,\pi\times\pi')=\prod_{\mathfrak{p}}L(s,\pi_{\mathfrak{p}}\times\pi_{\mathfrak{p}}')=\sum_{\mathfrak{a}}\frac{\lambda_{\pi\times\pi'}(\mathfrak{a})}{\mathrm{N}\mathfrak{a}^s},\qquad \mathfrak{q}_{\pi\times\pi'}=\prod_{\mathfrak{p}}\mathfrak{q}_{\pi_{\mathfrak{p}}\times\pi_{\mathfrak{p}}'}.
\]

When $v\mid\infty$, there are $n'n$ complex Langlands parameters $\mu_{j,j',\pi\times\pi'}(v)$ associated to $\pi_v$ and $\pi_v'$, from which one defines
\begin{align*}
L(s,\pi_{\infty}\times\pi_{\infty}') = \prod_{v\mid\infty}\prod_{j=1}^{n}\prod_{j'=1}^{n'}\Gamma_{v}(s+\mu_{j,j',\pi\times\pi'}(v))=\prod_{j=1}^{n'n[F:\mathbb{Q}]}\Gamma_{\mathbb{R}}(s+\mu_{\pi\times\pi'}(j)).
\end{align*}
The explicit descriptions of $\alpha_{j,j',\pi\times\pi'}(\mathfrak{p})$ and $\mu_{j,j',\pi\times\pi'}(v)$ in \cite{Humphries,ST} yield
\begin{equation}
\label{eqn:LRS_2}
|\alpha_{j,j',\pi\times\pi'}(\mathfrak{p})|\leq \mathrm{N}\mathfrak{p}^{\theta_n + \theta_{n'}},\qquad \mathrm{Re}(\mu_{\pi\times\pi'}(j))\geq -\theta_n - \theta_{n'}.
\end{equation}

Let $r_{\pi\times\pi'}=\mathbf{1}_{\widetilde{\pi}^*=\pi'^*}$.  When $(\pi,\pi')\in\mathfrak{F}_n^*\times\mathfrak{F}_{n'}^*$, the completed $L$-function
\begin{equation}
\label{eqn:Lambdaspixpi'}
\Lambda(s,\pi\times\pi')=(s(1-s))^{r_{\pi\times\pi'}}(D_F^{n'n}\mathrm{N}\mathfrak{q}_{\pi\times\pi'})^{\frac{s}{2}}L(s,\pi\times\pi')L(s,\pi_{\infty}\times\pi_{\infty}')
\end{equation}
is entire of order 1, and there exists a complex number $W(\pi\times\pi')$ of modulus 1 such that $\Lambda(s,\pi\times\pi')$ satisfies the functional equation $\Lambda(s,\pi\times\pi')=W(\pi\times\pi')\Lambda(1-s,\widetilde{\pi}\times\widetilde{\pi}')$.  When $\pi=\pi^*[it_{\pi}]\in\mathfrak{F}_n$ and $\pi'=\pi'^*[it_{\pi'}]\in\mathfrak{F}_{n'}$, we define
\begin{equation*}
\begin{gathered}
L(s,\pi_v\times\pi_v')=L(s+i(t_{\pi}+t_{\pi'}),\pi_v^*\times\pi_v'^*),\qquad L(s,\pi\times\pi')=L(s+i(t_{\pi}+t_{\pi'}),\pi^*\times\pi'^*),\\
\Lambda(s,\pi\times\pi')=\Lambda(s+i(t_{\pi}+t_{\pi'}),\pi^*\times\pi'^*).
\end{gathered}
\end{equation*}
In particular, $L(s,\pi\times\pi')$ is entire if and only if $\pi'^*\neq\widetilde{\pi}^*$ (i.e., $r_{\pi\times\pi'}=0$).  We have the crucial but straightforward identity
\begin{equation}
\label{eqn:identities}
L(s,\widetilde{\pi}\times\widetilde{\pi}')=\overline{L(\overline{s},\pi\times\pi')}.
\end{equation}

When $(\pi,\pi')\in\mathfrak{F}_n^*\times\mathfrak{F}_{n'}^*$ and $t\in\mathbb{R}$, we define $\mathfrak{C}_{\pi\times\pi'}= \mathfrak{C}_{\pi\times\pi'}(0)$, where
\[
\mathfrak{C}_{\pi\times\pi'}(it)= D_F^{n'n}\mathrm{N}\mathfrak{q}_{\pi\times\pi'}\prod_{v|\infty}\prod_{j=1}^n \prod_{j'=1}^{n'}(3+|it+\mu_{j,j',\pi\times\pi'}(v)|^{[F_v:\mathbb{R}]}).
\]
The work in \cite{BH} and \cite[Lemma~A.1]{Wattanawanichkul} (see also \cite[Appendix]{Humphries}) yields
\begin{equation}
\label{eqn:BH}
\mathfrak{C}_{\pi\times\pi'}(it)\leq \mathfrak{C}_{\pi\times\pi'}(3+|t|)^{n'n[F:\mathbb{Q}]},\qquad \mathfrak{C}_{\pi\times\pi'}\leq  D_F^{-n'n}\mathfrak{C}_{\pi}^{n'}\mathfrak{C}_{\pi'}^{n}.
\end{equation}
When $\pi=\pi^*[it_{\pi}]\in\mathfrak{F}_n$ and $\pi'=\pi'^*[it_{\pi'}]\in\mathfrak{F}_{n'}$, we define
\[
\mathfrak{C}_{\pi\times\pi'}(it) = \mathfrak{C}_{\pi^*\times\pi'^*}(i(t+t_{\pi}+t_{\pi'})),\qquad \mathfrak{C}_{\pi\times\pi'}=\mathfrak{C}_{\pi^*\times\pi'^*}(i(t_{\pi}+t_{\pi'})).
\]

Let $(\pi,\pi')\in\mathfrak{F}_n\times\mathfrak{F}_{n'}$.  Since $\Lambda(s,\pi\times\pi')$ is entire of order 1, there exist $A_{\pi\times\pi'},B_{\pi\times\pi'}\in\mathbb{C}$ such that $\Lambda(s,\pi\times\pi')$ has the Hadamard factorization
\begin{equation}
\label{eqn:hadamard}
\Lambda(s,\pi\times\pi')=e^{A_{\pi\times\pi'}+B_{\pi\times\pi'}s}\prod_{\Lambda(\rho,\pi\times\pi')=0}\Big(1-\frac{s}{\rho}\Big)e^{s/\rho}.
\end{equation}
The zeros $\rho$ in \eqref{eqn:hadamard} are the nontrivial zeros of $L(s,\pi\times\pi')$, and the zeros of $L(s,\pi\times\pi')$ that arise as poles of $(s+i(t_{\pi}+t_{\pi'}))^{r_{\pi\times\pi'}}L(s,\pi_{\infty}\times\pi_{\infty}')$ are the trivial zeros of $L(s,\pi\times\pi')$.

For an ideal $\mathfrak{a}$, we define
\[
\Lambda_{\pi\times\pi'}(\mathfrak{a})=\begin{cases}
	\displaystyle\sum_{j=1}^n \sum_{j'=1}^{n'}\alpha_{j,j',\pi\times\pi'}(\mathfrak{p})^{\ell}\log\mathrm{N}\mathfrak{p}&\mbox{if there exist $\mathfrak{p}$ and $\ell\in\mathbb{N}$ with $\mathfrak{a}=\mathfrak{p}^{\ell}$,}\\
0&\mbox{otherwise.}
\end{cases}
\]
If $\mathrm{Re}(s)>1$, then
\[
-\frac{L'}{L}(s,\pi\times\pi')=\sum_{\mathfrak{a}}\frac{\Lambda_{\pi\times\pi'}(\mathfrak{a})}{\mathrm{N}\mathfrak{a}^{s}},\qquad \log L(s,\pi\times\pi') = \sum_{\mathfrak{a}\neq\mathcal{O}_F}\frac{\Lambda_{\pi\times\pi'}(\mathfrak{a})}{\mathrm{N}\mathfrak{a}^{s}\log\mathrm{N}\mathfrak{a}}.
\]
The inequality
\begin{equation}
\label{eqn:Brumley_CS}
|\Lambda_{\pi\times\pi'}(\mathfrak{a})|\leq \sqrt{\Lambda_{\pi\times\widetilde{\pi}}(\mathfrak{a})\Lambda_{\pi'\times\widetilde{\pi}'}(\mathfrak{a})}\leq\frac{\Lambda_{\pi\times\widetilde{\pi}}(\mathfrak{a})+\Lambda_{\pi'\times\widetilde{\pi}'}(\mathfrak{a})}{2}
\end{equation}
was proved by Brumley \cite[Proposition A.1]{ST}.  It follows from \eqref{eqn:separate_dirichlet_coeffs} that
\begin{equation}
\label{eqn:separate_dirichlet_coeffs_2}
\Lambda_{\pi\times\pi'}(\mathfrak{p})=\lambda_{\pi}(\mathfrak{p})\lambda_{\pi'}(\mathfrak{p})\log\mathrm{N}\mathfrak{p},\qquad \mathfrak{p}\nmid\mathfrak{q}_{\pi}\mathfrak{q}_{\pi'}.
\end{equation}

\begin{lemma}[{\cite[Theorem~8.1]{JiangH}}]
\label{lem:HypH}
If $(\pi,\pi')\in\mathfrak{F}_n\times\mathfrak{F}_{n'}$, $\varepsilon>0$, and $x\geq (\mathfrak{C}_{\pi}\mathfrak{C}_{\pi'})^{\varepsilon}$, then
\[
\sum_{k=2}^{\infty}\sum_{\mathrm{N}\mathfrak{p}^k\leq x}|\Lambda_{\pi\times\pi'}(\mathfrak{p}^k)|\ll_{\varepsilon} x^{1-\frac{1}{\max\{n,n'\}^2+1}+\varepsilon}.
\]
\end{lemma}

\begin{lemma}[{\cite[Lemma~3.2]{HarcosThorner}}]
\label{lem:Li1}
For $(\pi,\pi')\in\mathfrak{F}_n\times\mathfrak{F}_{n'}$, consider the holomorphic function
\[
\mathscr{L}(s,\pi\times\pi')=\Big(\frac{s+i(t_{\pi}+t_{\pi'})-1}{s+i(t_{\pi}+t_{\pi'})+1}\Big)^{r_{\pi\times\pi'}}L(s,\pi\times\pi'),\qquad\mathrm{Re}(s)>-1.
\]
If $j\geq 0$, $\sigma\geq 0$, $t\in\mathbb{R}$, and $\varepsilon>0$, then $\mathscr{L}^{(j)}(\sigma+it,\pi\times\pi')\ll_{j,\varepsilon}\mathfrak{C}_{\pi\times\pi'}(it)^{\max(1-\sigma,0)/2+\varepsilon}$.
\end{lemma}

\subsection{Isobaric sums}
\label{subsec:isobaric}

To every $\ell$-tuple 
$(\pi_1,\dotsc,\pi_\ell)\in\mathfrak{F}_{n_1}\times\dotsb\times \mathfrak{F}_{n_\ell}$, the Langlands theory of Eisenstein series associates an automorphic representation of $\mathrm{GL}_{n_1+\dotsb+n_\ell}(\mathbb{A}_F)$, the isobaric sum $\Pi=\pi_1\boxplus\dotsb\boxplus\pi_\ell$. The contragredient and $L$-function of $\Pi$ are $\widetilde{\Pi}=\widetilde{\pi}_1\boxplus\dotsb\boxplus\widetilde{\pi}_{\ell}$ and $L(s,\Pi)=\prod_{j=1}^\ell L(s,\pi_j)$.  Let $\mathfrak{A}_n$ be the set of isobaric automorphic representations of $\mathrm{GL}_{n}(\mathbb{A}_F)$.  Given $\Pi=\pi_1\boxplus\dotsb\boxplus\pi_\ell\in\mathfrak{A}_n$ and $\Pi'=\pi_1'\boxplus\dotsb\boxplus\pi_m'\in\mathfrak{A}_{n'}$, define
\begin{equation*}
\begin{gathered}
L(s,\Pi\times\Pi')=\prod_{j=1}^\ell \prod_{k=1}^m L(s,\pi_j\times\pi_k'),\quad \mathfrak{C}_{\Pi\times\Pi'}(it) = \prod_{j=1}^{\ell} \prod_{k=1}^{m}\mathfrak{C}_{\pi_j\times\pi_{k}'}(it),\quad \mathfrak{C}_{\Pi\times\Pi'}=\mathfrak{C}_{\Pi\times\Pi'}(0),\\
\Lambda_{\Pi\times\Pi'}(\mathfrak{a})=\sum_{j=1}^{\ell}\sum_{k=1}^{m}\Lambda_{\pi_j\times\pi_{k}'}(\mathfrak{a}),\qquad -\frac{L'}{L}(s,\Pi\times\Pi')=\sum_{\mathfrak{a}}\frac{\Lambda_{\Pi\times\Pi'}(\mathfrak{a})}{\mathrm{N}\mathfrak{a}^s}.
\end{gathered}
\end{equation*}

\begin{lemma}[{\cite[Lemma~a]{HoffsteinRamakrishnan}}]
\label{lem:nonneg}
If $\Pi\in\mathfrak{A}_n$, then $L(s,\Pi\times\widetilde{\Pi})$ has nonnegative Dirichlet coefficients, as do $\log L(s,\Pi\times\widetilde{\Pi})$ and $-(L'/L)(s,\Pi\times\widetilde{\Pi})$.
\end{lemma}

\begin{lemma}[{\cite[Lemma 3.2]{HarcosThorner_Tatuzawa}}]
\label{lem:GHLnew}
Let $\Pi\in\mathfrak{A}_n$.  If $L(s,\Pi\times\widetilde{\Pi})$ has a pole of order $r\geq 1$ at $s=1$ and no other singularities, then $L(\sigma,\Pi\times\widetilde{\Pi})$ has at most $r$ nontrivial zeros $\beta+i\gamma$ (including multiplicity) satisfying $\beta \geq 1-1/(7r\log\mathfrak{C}_{\Pi\times\widetilde{\Pi}})$ and $|\gamma|\leq 1/(12\sqrt{r}\log\mathfrak{C}_{\Pi\times\widetilde{\Pi}})$.
\end{lemma}

Landau (see \cite[Theorem~5.28]{IK}) proved that if $j\in\{1,2\}$ and $\chi_j\pmod{q_j}$ is a primitive quadratic character whose $L$-function $L(s,\chi_j)$ has a greatest real zero $\beta_j$, then there exists a constant $c_{23}$ such that $\min\{\beta_1,\beta_2\}<1-c_{23}/\log(q_1 q_2)$.  Landau's result extends to standard $L$-functions \cite[Theorem~A]{HoffsteinRamakrishnan}.  Our proof of the nonvanishing result in Theorem~\ref{thm:main} uses a generalization of Landau's result.

\begin{proposition}
\label{prop:LandauPage}
Let $n_1,n_2,n_3\geq 1$ and $(\pi_1,\pi_2,\pi_3)\in\mathfrak{F}_{n_1}\times\mathfrak{F}_{n_2}\times\mathfrak{F}_{n_3}$.  Let $\mathfrak{C}$ be the maximum of $\mathfrak{C}_{\pi_1}$, $\mathfrak{C}_{\pi_2}$, and $\mathfrak{C}_{\pi_3}$.  If $L(s,\pi_1\times\widetilde{\pi}_2)$, $L(s,\pi_1\times\pi_3)$, and $L(s,\pi_2\times\pi_3)$ are entire, then $L(s,\pi_1\times\pi_3)L(s,\pi_2\times\pi_3)$ has at most one real zero (with multiplicity) in the interval
\begin{equation}
\label{eqn:LandauPage}
\Big[1-\frac{1}{126(n_1+n_2+n_3)\log\mathfrak{C}},1\Big).
\end{equation}
\end{proposition}

\begin{proof}
If $\Pi = \pi_1\boxplus \pi_2\boxplus \widetilde{\pi}_3$, then $L(s,\Pi\times\widetilde{\Pi})$ equals
\begin{align*}
&L(s,\pi_1\times\widetilde{\pi}_1)L(s,\pi_2\times\widetilde{\pi}_2)L(s,\pi_3\times\widetilde{\pi}_3)L(s,\pi_1\times\widetilde{\pi}_2)L(s,\widetilde{\pi}_1\times\pi_2)\\
&\cdot L(s,\pi_1\times\pi_3)L(s,\widetilde{\pi}_1\times\widetilde{\pi}_3) L(s,\pi_2\times\pi_3) L(s,\widetilde{\pi}_2\times\widetilde{\pi}_3).
\end{align*}
By hypothesis, the only singularity of $L(s,\Pi\times\widetilde{\Pi})$ is a pole of order $3$ at $s=1$.  Let $\mathfrak{C}=\max\{\mathfrak{C}_{\pi_1},\mathfrak{C}_{\pi_2},\mathfrak{C}_{\pi_3}\}$.  Suppose to the contrary that $L(s,\pi_1\times\pi_3)L(s,\pi_2\times\pi_3)$ has two real zeros $\beta_1$ and $\beta_2$ lying in \eqref{eqn:LandauPage} (with multiplicity if $\beta_1=\beta_2$).  Our hypotheses and \eqref{eqn:BH} ensure that $\log\mathfrak{C}_{\Pi\times\widetilde{\Pi}}\leq 2(n_1+n_2+n_3)\log(\mathfrak{C}_{\pi_1}\mathfrak{C}_{\pi_2}\mathfrak{C}_{\pi_3})\leq 6(n_1+n_2+n_3)\log\mathfrak{C}$. By Lemma~\ref{lem:GHLnew} applied to $\Pi$, $L(s,\Pi\times\widetilde{\Pi})$ has at most $3$ real zeros in the interval \eqref{eqn:LandauPage}.  However, by \eqref{eqn:identities}, $\beta_1$ and $\beta_2$ are also real zeros of $L(s,\widetilde{\pi}_1\times\widetilde{\pi}_3)L(s,\widetilde{\pi}_2\times\widetilde{\pi}_3)$.  Therefore, the assumption that $\beta_1$ and $\beta_2$ lie in \eqref{eqn:LandauPage} forces $L(s,\Pi\times\widetilde{\Pi})$ to have $4$ real zeros (including multiplicity) in \eqref{eqn:LandauPage}, contradicting Lemma~\ref{lem:GHLnew}.  The desired result follows.
\end{proof}

\subsection{Zeros of $L$-functions}

Many of these results were first proved for $(\pi,\pi')\in\mathfrak{F}_n^*\times\mathfrak{F}_{n'}^*$.  With  small adjustments, they apply to $(\pi,\pi')\in\mathfrak{F}_n\times\mathfrak{F}_{n'}$.

\begin{lemma}[cf.~{\cite[Proposition~5.7(1,2)]{IK}}]
\label{lem:log_deriv_strip}
Let $(\pi,\pi')\in\mathfrak{F}_n\times\mathfrak{F}_{n'}$.
\begin{enumerate}
	\item Let $t\in\mathbb{R}$.  The number $m_{\pi\times\pi'}(t)$ of zeros $\rho=\beta+i\gamma$ of $L(s,\pi\times\pi')$ such that $|\gamma-t|\leq 1$ satisfies $m_{\pi\times\pi'}(t)\ll\log\mathfrak{C}_{\pi\times\pi'}(it)$.
	\item If $s=\sigma+it$ and $-\frac{1}{2}\leq\sigma\leq 2$, then
	\begin{align*}
	&\frac{L'}{L}(s,\pi\times\pi')+\mathbf{1}_{\pi^*=\widetilde{\pi}'^*}\Big(\frac{1}{s+i(t_{\pi}+t_{\pi'})}+\frac{1}{s-1+i(t_{\pi}+t_{\pi'})}\Big)\\
	&-\sum_{\substack{j\in\{1,\ldots,n'n[F:\mathbb{Q}]\} \\ |s+\mu_{\pi\times\pi'}(j)|<1}}\frac{1}{s+\mu_{\pi\times\pi'}(j)}-\sum_{\substack{\Lambda(\rho,\pi\times\pi')=0 \\ |s-\rho|<1}}\frac{1}{s-\rho}\ll \log\mathfrak{C}_{\pi\times\pi'}(it).
	\end{align*}
\end{enumerate}
\end{lemma}

\begin{lemma}[cf.~{\cite[Lemma~9.2]{PascadiThorner}}]
\label{lem:log_deriv}
If $k\geq 1$ is an integer, $L(s,\pi\times\pi')$ is entire, and $s\in\mathbb{C}$ is not a zero of $L(s,\pi\times\pi')$, then
\[
\frac{(-1)^k}{k!}\Big(\frac{L'}{L}(s,\pi\times\pi')\Big)^{(k)}=\sum_{\Lambda(\rho,\pi\times\pi')=0}\frac{1}{(s-\rho)^{k+1}}-\sum_{j=1}^{n'n [F:\mathbb{Q}]}\sum_{m=0}^{\infty}\frac{1}{(2m+s+\mu_{\pi\times\pi'}(j))^{k+1}}.
\]
\end{lemma}

\begin{lemma}[cf.~{\cite[Lemma~9.3]{PascadiThorner}}]
\label{lem:Mertens}
If $\varepsilon>0$, then
\[
\sum_{\mathfrak{a}}\frac{|\Lambda_{\pi\times\pi'}(\mathfrak{a})|}{\mathrm{N}\mathfrak{a}^{1+\varepsilon}}\leq \frac{1}{\varepsilon}+\max\{n,n'\}\log\max\{\mathfrak{C}_{\pi},\mathfrak{C}_{\pi'}\}+O(1).
\]
\end{lemma}

\begin{lemma}[cf.~{\cite[Lemma~9.4]{PascadiThorner}}]
\label{lem:Linnik_lemma}
If $0<\varepsilon\leq 1$, then
\[
\sum_{\Lambda(\rho,\pi\times\pi')=0}\mathrm{Re}\Big(\frac{1}{1+\varepsilon-\rho}\Big)\leq \frac{1}{\varepsilon}+2\max\{n,n'\}\log\max\{\mathfrak{C}_{\pi},\mathfrak{C}_{\pi'}\}+O(1).
\]
Also, if $\sigma\geq 1$ and $n_{\pi\times\pi'}(\varepsilon,\sigma)=\#\{\rho\colon \Lambda(\rho,\pi\times\pi')=0,~|\sigma-\rho|\leq\varepsilon\}$, then
\[
n_{\pi\times\pi'}(\varepsilon,\sigma)\leq 8\varepsilon\max\{n,n'\}\log\max\{\mathfrak{C}_{\pi},\mathfrak{C}_{\pi'}\}+O(1).
\]
\end{lemma}

\begin{lemma}[cf. {\cite[Theorem~5.8]{IK}}]
\label{lem:LFZDE}
If $\sigma\geq 0$, $T\geq 3$, and
\begin{align*}
N_{\pi\times\pi'}(\sigma,T)=\#\{\rho=\beta+i\gamma\colon L(\rho,\pi\times\pi')=0,~\beta\geq\sigma,~|\gamma|\leq T\},
\end{align*}
then $N_{\pi\times\pi'}(\sigma,T)\ll T\log(\mathfrak{C}_{\pi}\mathfrak{C}_{\pi'}T)$.
\end{lemma}

\begin{lemma}
\label{lem:DH}
Let $L(s,\pi\times\pi')$ be entire with a greatest real zero $\beta_1\geq 0$.  Let $\Pi=\widetilde{\pi}\boxplus\pi'$.  There exists a constant $c_{24}\geq 2$ such that if $\beta+i\gamma\neq\beta_1$ is a zero of $L(s,\Pi\times\widetilde{\Pi})$, then
\[
\beta\leq 1-\frac{1}{c_{24}\log(\mathfrak{C}_{\pi}\mathfrak{C}_{\pi'} (|\gamma|+3))}\log\Big(\cfrac{1}{c_{24}(1-\beta_1)\log(\mathfrak{C}_{\pi}\mathfrak{C}_{\pi'} (|\gamma|+3))}\Big).
\]
\end{lemma}
\begin{proof}
This follows from Moreno~\cite[Proof~of~Theorem~4.2]{Moreno} applied to $\Pi = \widetilde{\pi}\boxplus\pi'$.  The difference between Moreno's definition of analytic conductor and ours changes only the dependence of $c_{24}$ on $n$, $n'$, and $F$.
\end{proof}

\begin{lemma}
\label{lem:LFZDE2}
Let $\Pi=\widetilde{\pi}\boxplus\pi'$.  If $L(s,\pi\times\pi')$ is entire and has a greatest real zero $\beta_{1}\geq 0$, $\sigma\geq 0$, $T\geq 3$, and
\[
N_*(\sigma,T,\Pi\times\widetilde{\Pi})=\#\{\rho=\beta+i\gamma\colon \rho\neq\beta_1,~L(\rho,\Pi\times\widetilde{\Pi})=0,~\beta\geq\sigma,~|\gamma|\leq T\},
\]
then $N_*(\sigma,T,\Pi\times\widetilde{\Pi})\ll (1-\beta_1)(\log\mathfrak{C}_{\pi}\mathfrak{C}_{\pi'})^2 T^{c_{24}+2}(\mathfrak{C}_{\pi}\mathfrak{C}_{\pi'})^{c_{24}(1-\sigma)}$.
\end{lemma}

\begin{proof}
If $c_{24}(1-\beta_1)\log(\mathfrak{C}_{\pi}\mathfrak{C}_{\pi'} (T+3))\geq 1$, then the result is trivial compared to Lemma~\ref{lem:LFZDE}.  If not, then by Lemma~\ref{lem:DH}, $N_*(\sigma,T,\Pi\times\widetilde{\Pi})=0$ when
\[
\sigma\geq 1-\frac{1}{c_{24}\log(\mathfrak{C}_{\pi}\mathfrak{C}_{\pi'} (T+3))}\log\Big(\cfrac{1}{c_{24}(1-\beta_1)\log(\mathfrak{C}_{\pi}\mathfrak{C}_{\pi'} (T+3))}\Big).
\]
If $\sigma$ is smaller, then $1\ll (1-\beta_{1})\log(\mathfrak{C}_{\pi}\mathfrak{C}_{\pi'}T)(\mathfrak{C}_{\pi}\mathfrak{C}_{\pi'}T)^{c_{24}(1-\sigma)}$, in which case
\begin{align*}
N_*(\sigma,T,\Pi\times\widetilde{\Pi})&\ll (1-\beta_{1})\log(\mathfrak{C}_{\pi}\mathfrak{C}_{\pi'}T)(\mathfrak{C}_{\pi}\mathfrak{C}_{\pi'}T)^{c_{24}(1-\sigma)}\\
&\cdot (N_{\pi\times\widetilde{\pi}}(\sigma,T)+N_{\pi'\times\widetilde{\pi}'}(\sigma,T)+N_{\pi\times\pi'}(\sigma,T)+N_{\widetilde{\pi}\times\widetilde{\pi}'}(\sigma,T)).
\end{align*}
The desired result now follows from Lemma~\ref{lem:LFZDE}.
\end{proof}

\section{Zero repulsion and sums over prime ideals}
\label{sec:partial_sums}

We use Lemmas~\ref{lem:GHLnew} and \ref{lem:LFZDE2} to prove the following proposition.

\begin{proposition}
\label{prop:Linnik}
Let $(\pi,\pi')\in\mathfrak{F}_n\times\mathfrak{F}_{n'}$. If $c_{25}\geq\max\{8c_{24},8n^3,8n'^3\}$, $x\geq (\mathfrak{C}_{\pi}\mathfrak{C}_{\pi'})^{c_{25}}$, $\widetilde{\pi}^*\neq\pi'^*$, and $L(s,\pi\times\pi')$ has a greatest real zero
\begin{equation}
\label{eqn:lower_beta_1}
\beta_{1}> 1-\frac{1}{126(2n+n')\log\max\{\mathfrak{C}_{\pi},\mathfrak{C}_{\pi'}\}},
\end{equation}
then
\begin{equation}
\label{eqn:primesum}
\sum_{x<\mathrm{N}\mathfrak{p}\leq x^{12000}}\frac{|\lambda_{\widetilde{\pi}}(\mathfrak{p})+\lambda_{\pi'}(\mathfrak{p})|^2\log\mathrm{N}\mathfrak{p}}{\mathrm{N}\mathfrak{p}}\ll (1-\beta_{1})(\log x)^3.
\end{equation}
\end{proposition}

\begin{remark}
Several numerical constants appearing in our arguments, including $12000$, are chosen only for definiteness and not optimized.
\end{remark}

The proof has three components.  First, we use Lemma~\ref{lem:nonneg} to bound the sum over prime ideals by a smoothed sum involving the Dirichlet coefficients of $-\frac{L'}{L}(s,\Pi\times\widetilde{\Pi})$, where $\Pi=\widetilde{\pi}\boxplus\pi'$.  Second, the explicit formula for this sum isolates contributions from the pole at $s=1$ and the exceptional pair $\beta_1$ and $1-\beta_1$, which we handle individually.  Finally, Lemma~\ref{lem:LFZDE2} bounds the contributions from the remaining zeros with the decisive factor of $1-\beta_1$.

\begin{proof}
Let $(\pi,\pi')\in\mathfrak{F}_n\times\mathfrak{F}_{n'}$ satisfy $\widetilde{\pi}^*\neq\pi'^*$.  Let $\Pi=\widetilde{\pi}\boxplus\pi'$, and define
\[
S(x;\Pi)=\sum_{x<\mathrm{N}\mathfrak{a}\leq x^{12000}}\frac{\Lambda_{\Pi\times\widetilde{\Pi}}(\mathfrak{a})}{\mathrm{N}\mathfrak{a}},\qquad x\geq (\mathfrak{C}_{\pi}\mathfrak{C}_{\pi'})^{c_{25}}.
\]
Note that $L(s,\Pi\times\widetilde{\Pi})=L(s,\pi\times\widetilde{\pi})L(s,\pi'\times\widetilde{\pi}')L(s,\pi\times\pi')L(s,\widetilde{\pi}\times\widetilde{\pi}')$.  If $\mathrm{N}\mathfrak{p}>x$, then $\mathfrak{p}\nmid\mathfrak{q}_{\pi}\mathfrak{q}_{\pi'}$ and $|\lambda_{\widetilde{\pi}}(\mathfrak{p})+\lambda_{\pi'}(\mathfrak{p})|^2 \log\mathrm{N}\mathfrak{p} = \Lambda_{\Pi\times\widetilde{\Pi}}(\mathfrak{p})$ by \eqref{eqn:separate_dirichlet_coeffs_2}.  Since $\Lambda_{\Pi\times\widetilde{\Pi}}(\mathfrak{a})\geq 0$ for all $\mathfrak{a}$ by Lemma~\ref{lem:nonneg}, it follows that the left-hand side of \eqref{eqn:primesum} is at most $S(x;\Pi)$.

Since $\widetilde{\pi}^*\neq\pi'^*$, $L(s,\pi\times\pi')L(s,\widetilde{\pi}\times\widetilde{\pi}')$ is entire, and the only singularity of $L(s,\Pi\times\widetilde{\Pi})$ is a pole of order $2$ at $s=1$.  Therefore, under our hypothesis for $\beta_{1}$ (which is also the greatest real zero of $L(s,\widetilde{\pi}\times\widetilde{\pi}')$ by \eqref{eqn:identities}), \eqref{eqn:BH} and  Lemma~\ref{lem:GHLnew} imply that any nontrivial zero $\rho\neq\beta_{1}$ of $L(s,\Pi\times\widetilde{\Pi})$ satisfies $1\ll|\rho-1|\log \mathfrak{C}_{\pi}\mathfrak{C}_{\pi'}$.  By the functional equation for $L(s,\Pi\times\widetilde{\Pi})$ and the uniqueness of $\beta_{1}$, $1-\beta_{1}$ is a nontrivial zero of $L(s,\Pi\times\widetilde{\Pi})$ of order 2, and any nontrivial zero $\rho\neq 1-\beta_{1}$ satisfies $1\ll |\rho|\log\mathfrak{C}_{\pi}\mathfrak{C}_{\pi'}$.  We conclude that
\begin{equation}
\label{eqn:first_lower}
1\ll|\rho|\cdot|\rho-1|\log\mathfrak{C}_{\pi}\mathfrak{C}_{\pi'},\qquad \rho\notin\{\beta_{1},1-\beta_{1}\}.
\end{equation}

Let $n_{0}=\max\{n,n'\}$.  The nonnegativity of $\Lambda_{\Pi\times\widetilde{\Pi}}(\mathfrak{a})$ (Lemma~\ref{lem:nonneg}), Mellin inversion, Lemma~\ref{lem:log_deriv_strip}, the residue theorem, and \eqref{eqn:LRS_2} together yield
\begin{equation}
\label{eqn:S_expansion}
\begin{aligned}
S(x;\Pi)&\ll \sum_{\mathfrak{a}}\frac{\Lambda_{\Pi\times\widetilde{\Pi}}(\mathfrak{a})}{\mathrm{N}\mathfrak{a}}(e^{-\mathrm{N}\mathfrak{a}/x^{12000}}-e^{-\mathrm{N}\mathfrak{a}/x})\\
&=\frac{1}{2\pi i}\int_{3-i\infty}^{3+i\infty}-\frac{L'}{L}(s+1,\Pi\times\widetilde{\Pi})(x^{12000s}-x^s)\Gamma(s)ds\\
&=23998\log x - \sum_{\Lambda(\rho,\Pi\times\widetilde{\Pi})=0}(x^{12000(\rho-1)}-x^{\rho-1})\Gamma(\rho-1)+O(x^{-\frac{1}{n_{0}^2+1}}).
\end{aligned}\hspace{-2.58mm}
\end{equation}
It follows from our range of $x$ and \eqref{eqn:Brumley_RS} that the error term is $O(1-\beta_{1})$.

Using Stirling's formula, \eqref{eqn:Brumley_RS}, and our range of $x$, we isolate the contribution from the main term, the nontrivial zero $\beta_{1}$, and the nontrivial zero $1-\beta_{1}$.  Both $\beta_1$ and $1-\beta_1$ are nontrivial zeros of $L(s,\Pi\times\widetilde{\Pi})$ with multiplicity $2$.  The main term and $\beta_1$ contribute
\[
23998\log x - 2(x^{12000(\beta_{1}-1)}-x^{\beta_{1}-1})\Gamma(\beta_{1}-1)\ll (1-\beta_{1})(\log x)^2,
\]
while $1-\beta_{1}$ contributes $(x^{-12000\beta_{1}}-x^{-\beta_{1}})\Gamma(-\beta_{1})\ll (1-\beta_1)^{-1}x^{-\beta_1}\ll 1-\beta_1$.

Let $\rho=\beta+i\gamma\notin\{\beta_{1},1-\beta_{1}\}$ satisfy $\Lambda(\rho,\Pi\times\widetilde{\Pi})=0$.  By Lemma~\ref{lem:GHLnew}, \eqref{eqn:lower_beta_1}, and the functional equation for $\Lambda(s,\Pi\times\widetilde{\Pi})$, we find that $|\rho|\gg (\log\mathfrak{C}_{\pi}\mathfrak{C}_{\pi'})^{-1}$.  This lower bound, Stirling's formula, and \eqref{eqn:first_lower} show that if $m\geq 0$ is an integer such that $m\leq|\gamma|<m+3$, then $|\Gamma(\rho-1)|\ll e^{-m}\log\mathfrak{C}_{\pi}\mathfrak{C}_{\pi'}\ll e^{-m}\log x$.  Now, Lemma~\ref{lem:LFZDE2} yields
\begin{align*}
&\sum_{m=0}^{\infty}\sum_{\substack{\rho\notin\{\beta_{1},1-\beta_{1}\} \\ m\leq |\gamma|<m+3}}(x^{12000(\rho-1)}-x^{\rho-1})\Gamma(\rho-1)\ll \sum_{m=0}^{\infty}\frac{\log x}{e^m}\sum_{\substack{\rho\notin\{\beta_{1},1-\beta_{1}\} \\ m\leq |\gamma|<m+3}}x^{\beta-1}\\
&\ll (\log x)^2\sum_{m=0}^{\infty}e^{-m}\int_0^1 N_*(\sigma,m+3,\Pi\times\widetilde{\Pi})x^{\sigma-1}d\sigma\\
&\ll (1-\beta_{1})(\log x)^4\sum_{m=0}^{\infty}e^{-m}(m+3)^{c_{24}+2}\int_0^1 x^{\frac{\sigma-1}{2}}d\sigma\ll (1-\beta_{1})(\log x)^3.
\end{align*}
The proposition now follows.
\end{proof}

\section{Proof of Theorem~\ref{thm:main} (nonvanishing)}
\label{sec:start}

In this section, we prove the nonvanishing component of Theorem~\ref{thm:main}.

\begin{proposition}
\label{prop:nonvanish}
Fix $\pi'\in\mathfrak{F}_{n'}$.  For all $\varepsilon>0$, there exists an ineffective constant $c_{26}=c_{26}(n,F,\pi',\varepsilon)>0$ such that if $\pi\in\mathfrak{F}_n$ and $\sigma\geq 1-c_{26}\mathfrak{C}_{\pi}^{-\varepsilon}$, then $L(\sigma,\pi\times\pi')\neq 0$.
\end{proposition}

\subsection{Initial reductions}

Let $F$ be a number field, $n,n'\geq 1$, and $\pi'\in\mathfrak{F}_{n'}$.  Let $0<\eta<1/(100n'n)^2$.   Write $\pi=\pi^*[it_{\pi}]$ and $\pi'=\pi'^*[it_{\pi'}]$.  If $\pi^*\in\{\pi'^*,\widetilde{\pi}'^*\}$, then
\[
L(\sigma,\pi\times\pi') \in \{L(\sigma,\pi'\times(\pi'^*[it_{\pi}])),L(\sigma,\pi'\times(\widetilde{\pi}'^*[it_{\pi}]))\}.
\] 
In either case, \eqref{eqn:HarcosThorner_RS} applied to $\chi=|\cdot|^{it_{\pi}}$ implies that there exists an ineffective constant $c_{27}=c_{27}(n,F,\pi',\eta)>0$ such that if $\sigma\geq 1-c_{27}(|t_{\pi}|+3)^{-\eta}$, then $L(\sigma,\pi\times\pi')\neq 0$.  Proposition~\ref{prop:nonvanish} now follows from Lemma~\ref{lem:t_pi_bound}.  We may now assume that $\pi^*\notin\{\pi'^*,\widetilde{\pi}'^*\}$.  If $\pi^*\notin\{\pi'^*,\widetilde{\pi}'^*\}$ and $\mathfrak{C}_{\pi'}\geq\mathfrak{C}_{\pi}$, then Proposition~\ref{prop:nonvanish} trivially follows from \eqref{eqn:Brumley_RS}.  We may now assume that $\pi\in \mathcal{F}_n(\pi')=\{\pi\in\mathfrak{F}_n\colon \pi^*\neq\pi'^*,~\widetilde{\pi}^*\neq\pi'^*,~\mathfrak{C}_{\pi}>\mathfrak{C}_{\pi'}\}$.

If, for all $\pi\in\mathcal{F}_n(\pi')$, $L(s,\pi\times\pi')$ has no zero in the interval $(1-\eta,1)$, then Proposition~\ref{prop:nonvanish} holds trivially.  Otherwise, there exists $\pi_{\eta}=\pi_{\eta}^*[it_{\pi_{\eta}}]\in\mathcal{F}_n(\pi')$ (depending at most on $n$, $F$, $\pi'$, and $\eta$) such that $L(s,\pi_{\eta}\times\pi')$ has a real zero $\beta_{\eta}$ satisfying $1-\eta<\beta_{\eta}<1$.

If $\pi^*\in\{\pi_{\eta}^*,\widetilde{\pi}_{\eta}^*\}$, then $L(\sigma,\pi\times\pi')\in\{L(\sigma,\pi_{\eta}^*\times\pi'[it_{\pi}]),L(\sigma,\widetilde{\pi}_{\eta}^*\times\pi'[it_{\pi}])\}$.  Therefore, \eqref{eqn:HarcosThorner_RS} with $\chi=|\cdot|^{it_{\pi}}$ implies that there exists an ineffective constant $c_{28}=c_{28}(n,F,\pi',\eta)>0$ such that if $\sigma\geq 1-c_{28}(|t_{\pi}|+3)^{-\eta}$, then  $L(\sigma,\pi\times\pi')\neq 0$.  Now, Lemma~\ref{lem:t_pi_bound} implies Proposition~\ref{prop:nonvanish}.  Henceforth, we assume that $\pi^*\notin\{\pi_{\eta}^*,\widetilde{\pi}_{\eta}^*\}$.  We may also assume that
\begin{equation}
\label{eqn:lower_bound_AC}
\mathfrak{C}_{\pi}>\max\Big\{\mathfrak{C}_{\pi'},\mathfrak{C}_{\pi_{\eta}},\exp\Big(\frac{1}{126(2n+n')(1-\beta_{\eta})}\Big)\Big\}.
\end{equation}
(Otherwise, $\mathfrak{C}_{pi}$ is bounded in terms of $n$, $F$, $\pi'$, and $\eta$, so Proposition~\ref{prop:nonvanish} follows from \eqref{eqn:Brumley_RS}.)  Thus, the hypotheses of Proposition~\ref{prop:LandauPage} are satisfied with $(\pi_1,\pi_2,\pi_3)=(\pi,\pi_{\eta},\pi')$.

If $L(\sigma,\pi\times\pi')\neq 0$ for  $\sigma\in (1-\eta,1)$, then Proposition~\ref{prop:nonvanish} holds.  Now, assume that $L(s,\pi\times\pi')$ has a greatest real zero $\beta_{1}\in(1-\eta,1)$.  If $\beta_{1}=\beta_{\eta}$, then $\beta_{1}$ is a zero of $L(s,\pi\times\pi')L(s,\pi_{\eta}\times\pi')$ of multiplicity $2$.  Thus, by Proposition~\ref{prop:LandauPage}, $\beta_{1}$ cannot lie in
\begin{equation}
\label{eqn:LandauPage_2}
\Big(1-\frac{1}{126(2n+n')\log\mathfrak{C}_{\pi}},1\Big),
\end{equation}
and Proposition~\ref{prop:nonvanish} is satisfied.  We may now assume that $\beta_{1}\neq\beta_{\eta}$.  If $\beta_{1}$ does not lie in \eqref{eqn:LandauPage_2}, then Proposition~\ref{prop:nonvanish} holds, since $1-\beta_1\gg (\log\mathfrak{C}_{\pi})^{-1}\gg_{\eta}\mathfrak{C}_{\pi}^{-\eta}$.  Otherwise, it follows by Proposition~\ref{prop:LandauPage}, the bound $1-\eta<\beta_{\eta}<1$, and \eqref{eqn:lower_bound_AC} that
\begin{equation}
\label{eqn:zeros_trapped}
1-\eta<\beta_{\eta}<1-\frac{1}{126(2n+n')\log\mathfrak{C}_{\pi}}<\beta_{1}.
\end{equation}

\subsection{Lower bounds}
\label{sec:lower_G}
 
We introduce $n_0=\max\{n,n'\}$, $\mathcal{L}=8n_0^2\log\mathfrak{C}_{\pi}$, and $s_0=1+\eta$.  It follows from \eqref{eqn:zeros_trapped} that $\eta \mathcal{L}\gg 1$.  For an integer $k\geq 1$, define
\[
\mathscr{F}(z) := L(z,\pi_{\eta}\times\pi')L(z,\pi_{\eta}\times \widetilde{\pi}),\qquad G_k(z):=\frac{(-1)^{k}}{k!}\Big(\frac{\mathscr{F}'}{\mathscr{F}}\Big)^{(k)}(z).
\]
By two applications of Lemma~\ref{lem:log_deriv}, we see that
\begin{equation}
\label{eqn:Gk_expansion}
\begin{aligned}
G_k(s_0)&=\sum_{\Lambda(\rho,\pi_{\eta}\times\pi')=0}\frac{1}{(s_0-\rho)^{k+1}}-\sum_{j=1}^{n' n [F:\mathbb{Q}]}\sum_{m=0}^{\infty}\frac{1}{(2m+s_0+\mu_{\pi_{\eta}\times\pi'}(j))^{k+1}}\\
&+\sum_{\Lambda(\rho,\pi_{\eta}\times\widetilde{\pi})=0}\frac{1}{(s_0-\rho)^{k+1}}-\sum_{j=1}^{n^2 [F:\mathbb{Q}]}\sum_{m=0}^{\infty}\frac{1}{(2m+s_0+\mu_{\pi_{\eta}\times\widetilde{\pi}}(j))^{k+1}}.
\end{aligned}
\end{equation}
We have $L(\beta_1,\pi_{\eta}\times\pi')\neq 0$ by \eqref{eqn:zeros_trapped} and Proposition~\ref{prop:LandauPage} applied to $(\pi_1,\pi_2,\pi_3)=(\pi,\pi_{\eta},\pi')$.  Also, $L(\beta_1,\widetilde{\pi}_{\eta}\times\pi)\neq 0$ by Proposition~\ref{prop:LandauPage} applied to $(\pi_1,\pi_2,\pi_3)=(\pi',\widetilde{\pi}_{\eta},\pi)$.  Since $L(s,\widetilde{\pi}_{\eta}\times\pi)$ and $L(s,\pi_{\eta}\times\widetilde{\pi})$ have the same real zeros, we conclude that $L(\beta_1,\pi_{\eta}\times\widetilde{\pi})\neq 0$. Thus, neither sum over $\rho$ in \eqref{eqn:Gk_expansion} contains $\beta_1$, and the first sum contains $\beta_{\eta}$.

By \eqref{eqn:LRS_2}, the trivial zeros contribute $\ll(n_0^2)^{k+1}\ll (200\eta)^{-k-1}$. Lemma~\ref{lem:Linnik_lemma} implies
\begin{align*}
&\sum_{\substack{\Lambda(\rho,\pi_{\eta}\times\pi')=0 \\ |1-\rho|>200\eta}}\frac{1}{|s_0-\rho|^{k+1}}+\sum_{\substack{\Lambda(\rho,\pi_{\eta}\times \widetilde{\pi})=0 \\ |1-\rho|>200\eta}}\frac{1}{|s_0-\rho|^{k+1}}\\
&\leq \frac{1}{(200\eta)^{k-1}}\Big(\sum_{\substack{\Lambda(\rho,\pi_{\eta}\times\pi')=0}}\frac{1}{|s_0-\rho|^{2}}+\sum_{\substack{\Lambda(\rho,\pi_{\eta}\times \widetilde{\pi})=0}}\frac{1}{|s_0-\rho|^{2}}\Big) \ll\frac{\eta\mathcal{L}}{(200\eta)^{k+1}}.
\end{align*}
In summary, since $\eta\mathcal{L}\gg 1$ by \eqref{eqn:zeros_trapped}, there exists a constant $c_{29}\geq 2$ such that
\begin{equation}
\label{eqn:G_asymptotic}
\Big|G_k(s_0)-\Big(\sum_{\substack{\Lambda(\rho,\pi_{\eta}\times\pi')=0 \\ \rho\neq\beta_1 \\ |1-\rho|\leq 200\eta}}\frac{1}{(s_0-\rho)^{k+1}}+\sum_{\substack{\Lambda(\rho,\pi_{\eta}\times\widetilde{\pi})=0 \\ \rho\neq\beta_1 \\ |1-\rho|\leq 200\eta}}\frac{1}{(s_0-\rho)^{k+1}}\Big)\Big|\leq\frac{c_{29}\eta\mathcal{L}}{(200\eta)^{k+1}}.
\end{equation}

\begin{proposition}
\label{prop:lower_bound_high_deriv}
If \eqref{eqn:zeros_trapped} is true, then there exists a constant $c_{30}\geq 2$ such that if \begin{equation}
\label{eqn:K_pre}
K\geq 1600\eta\mathcal{L}+c_{30},
\end{equation}
then there exists an integer $k\in[K,2K+1]$ such that $|\eta^{k+1}G_{k}(s_0)|\geq \frac{1}{2}(100)^{-k-1}$.
\end{proposition}

\begin{proof}
By \eqref{eqn:zeros_trapped}, the zero $\beta_{\eta}$ of $L(s,\pi_{\eta}\times\pi')$ is in the domain of summation in the first sum over zeros in \eqref{eqn:G_asymptotic}.  Therefore, the union of the two sums over zeros in \eqref{eqn:G_asymptotic} is nonempty.  By Lemma~\ref{lem:Linnik_lemma}, there exists a constant $c_{30}\geq 2$ such that there are at most $1600\eta\mathcal{L}+c_{30}$ terms in these two sums.  We now apply Proposition~\ref{prop:Turan} to the union of the two sums over zeros.  We choose $z_1 = (s_0-\beta_{\eta})^{-1}$.  The lower bound $|z_1|\geq 1/(2\eta)$ follows from \eqref{eqn:zeros_trapped}.  Therefore, if $K\geq 1600\eta\mathcal{L}+c_{30}$, then by Proposition~\ref{prop:Turan}, there exists an integer $k$ with $k+1\in[K+1,2(K+1)]$ such that if $c_{30}$ is sufficiently enlarged in terms of $c_{29}$, then
\[
|\eta^{k+1}G_k(s_0)|\geq \Big(\frac{1}{100}\Big)^{k+1}\Big(1-c_{29}\frac{\eta\mathcal{L}}{2^{k+1}}\Big)\geq \Big(\frac{1}{100}\Big)^{k+1}\Big(1-c_{29}\frac{k}{2^{k+1}}\Big)\geq  \frac{1}{2(100)^{k+1}}.\qedhere
\]
\end{proof}

\subsection{Upper bounds}
\label{sec:upper_G}
We now establish an upper bound for $|\eta^{k+1}G_k(s_0)|$.

\begin{proposition}
\label{prop:upper_bound_high_deriv}
There exists a constant $c_{31}\geq 2$ such that if $0<\eta<1/(100n'n)^2$, $s_0=1+\eta$, $K\geq 1$, $k\in[K,2K+1]$ is an integer, and we define
\begin{equation}
\label{eqn:Neta_def}
N_{\eta}=\exp(K/(300\eta))>\mathrm{N}\mathfrak{q}_{\pi}\mathfrak{q}_{\pi_{\eta}}\mathfrak{q}_{\pi'},\qquad N_{\eta}^*=N_{\eta}^{12000},
\end{equation}
then
\[
|\eta^{k+1}G_k(s_0)|\leq \eta^2 \int_{N_{\eta}}^{N_{\eta}^*}\Big|\sum_{\mathrm{N}\mathfrak{p}\in[N_{\eta},u]}\frac{\lambda_{\pi_{\eta}}(\mathfrak{p})(\lambda_{\pi'}(\mathfrak{p})+\lambda_{\widetilde{\pi}}(\mathfrak{p}))\log\mathrm{N}\mathfrak{p}}{\mathrm{N}\mathfrak{p}}\Big|\frac{du}{u}+\frac{c_{31}k}{110^k}.
\]
\end{proposition}
\begin{proof}
Let $k\in[K,2K+1]$ be an integer, and define $j_k(u)=e^{-u}u^k/k!$.  We have expansion
\begin{equation}
\label{eqn:DirichletExpansion}
|\eta^{k+1}G_k(s_0)|=\eta\Big|\sum_{\mathfrak{a}}\frac{\Lambda_{\pi_{\eta}\times\pi'}(\mathfrak{a})+\Lambda_{\pi_{\eta}\times\widetilde{\pi}}(\mathfrak{a})}{\mathrm{N}\mathfrak{a}}j_k(\eta\log \mathrm{N}\mathfrak{a})\Big|.
\end{equation}
We split the sum depending on whether $\mathrm{N}\mathfrak{a}\in[N_{\eta},N_{\eta}^*]$ or not.  By partial summation in the former range and the triangle inequality in the latter, \eqref{eqn:DirichletExpansion} is
\begin{equation}
\label{eqn:break_up_sum}
\begin{aligned}
&\leq \eta\Big(\sum_{\mathrm{N}\mathfrak{a}\notin[N_{\eta},N_{\eta}^*]}\frac{|\Lambda_{\pi_{\eta}\times\pi'}(\mathfrak{a})+\Lambda_{\pi_{\eta}\times\widetilde{\pi}}(\mathfrak{a})|}{\mathrm{N}\mathfrak{a}}j_k(\eta\log \mathrm{N}\mathfrak{a})\\
&+\sum_{\mathrm{N}\mathfrak{a}\in [N_{\eta},N_{\eta}^*]}\frac{|\Lambda_{\pi_{\eta}\times\pi'}(\mathfrak{a})+\Lambda_{\pi_{\eta}\times\widetilde{\pi}}(\mathfrak{a})|}{\mathrm{N}\mathfrak{a}}j_k(\eta\log N_{\eta}^*)\\
&+\int_{N_{\eta}}^{N_{\eta}^*}\Big|\frac{d}{du}j_k(\eta\log u)\Big|\cdot\Big|\sum_{\mathrm{N}\mathfrak{a}\in[N_{\eta},u]}\frac{\Lambda_{\pi_{\eta}\times\pi'}(\mathfrak{a})+\Lambda_{\pi_{\eta}\times\widetilde{\pi}}(\mathfrak{a})}{\mathrm{N}\mathfrak{a}}\Big|du\Big).
\end{aligned}
\end{equation}

First, we handle the contribution from $\mathrm{N}\mathfrak{a}\notin[N_{\eta},N_{\eta}^*]$.  If $\mathrm{N}\mathfrak{a}\leq N_{\eta}$, then $\eta\log \mathrm{N}\mathfrak{a}\leq K/300$.  Since $k\in[K,2K+1]$ and $k!\geq (k/e)^k$, we observe that
\begin{equation}
\label{eqn:jk1}
j_k(\eta\log \mathrm{N}\mathfrak{a})=\frac{\mathrm{N}\mathfrak{a}^{-\eta}(\eta\log \mathrm{N}\mathfrak{a})^k}{k!}\leq \mathrm{N}\mathfrak{a}^{-\eta}\Big(\frac{e\eta\log \mathrm{N}\mathfrak{a}}{k}\Big)^k\leq \mathrm{N}\mathfrak{a}^{-\eta/2}110^{-k},\quad \mathrm{N}\mathfrak{a}\leq N_{\eta}.\hspace{-1mm}
\end{equation}
If $\mathrm{N}\mathfrak{a}\geq N_{\eta}^*$, then $\eta\log \mathrm{N}\mathfrak{a}\geq 40K$.  Since $e^{-u/2}u^k/k!$ is decreasing for $u>2k$, we find that
\begin{equation}
\label{eqn:jk2}
j_k(\eta\log \mathrm{N}\mathfrak{a})=\frac{e^{-\frac{1}{2}\eta\log \mathrm{N}\mathfrak{a}}(\eta\log \mathrm{N}\mathfrak{a})^k}{k!\mathrm{N}\mathfrak{a}^{\frac{\eta}{2}}}\leq \frac{e^{-20K}(40K)^k}{k!\mathrm{N}\mathfrak{a}^{\frac{\eta}{2}}}\leq \mathrm{N}\mathfrak{a}^{-\frac{\eta}{2}}110^{-k},\quad \mathrm{N}\mathfrak{a}\geq N_{\eta}^*.
\end{equation}
It also follows from \eqref{eqn:jk2} that $j_k(\eta\log N_{\eta}^*)\ll \mathrm{N}\mathfrak{a}^{-\eta/2}110^{-k}$ for $\mathrm{N}\mathfrak{a}\in[N_{\eta},N_{\eta}^*]$.  Applying these to $j_k(\eta\log \mathrm{N}\mathfrak{a})$, we conclude via Lemma~\ref{lem:Mertens} that the first two sums in \eqref{eqn:break_up_sum} are
\[
\ll\frac{\eta}{110^k}\sum_{\mathfrak{a}}\frac{|\Lambda_{\pi_{\eta}\times\pi'}(\mathfrak{a})|+|\Lambda_{\pi_{\eta}\times\widetilde{\pi}}(\mathfrak{a})|}{\mathrm{N}\mathfrak{a}^{1+\eta/2}}\ll\frac{\eta\mathcal{L}}{110^k}\ll \frac{k}{110^k}.
\]

Note that $|\frac{d}{du}j_k(\eta\log u)|=|j_{k-1}(\eta\log u)-j_k(\eta\log u)|(\eta/u)\leq \eta/u$.  By this estimate, our range of $\eta$, \eqref{eqn:separate_dirichlet_coeffs_2}, Lemma~\ref{lem:HypH}, and \eqref{eqn:Neta_def}, the third line of \eqref{eqn:break_up_sum} is
\begin{align*}
&\leq \eta^2 \int_{N_{\eta}}^{N_{\eta}^*}\Bigg|\sum_{\mathrm{N}\mathfrak{a}\in[N_{\eta},u]}\frac{\Lambda_{\pi_{\eta}\times\pi'}(\mathfrak{a})+\Lambda_{\pi_{\eta}\times\widetilde{\pi}}(\mathfrak{a})}{\mathrm{N}\mathfrak{a}}\Bigg|\frac{du}{u}\\
&= \eta^2 \int_{N_{\eta}}^{N_{\eta}^*}\Bigg|\sum_{\mathrm{N}\mathfrak{p}\in[N_{\eta},u]}\frac{\lambda_{\pi_{\eta}}(\mathfrak{p})(\lambda_{\pi'}(\mathfrak{p})+\lambda_{\widetilde{\pi}}(\mathfrak{p}))\log\mathrm{N}\mathfrak{p}}{\mathrm{N}\mathfrak{p}}\Bigg|\frac{du}{u}+O\Big(\frac{k}{110^k}\Big).
\end{align*}
The proposition now follows.
\end{proof}

\subsection{Finishing the proof of Proposition~\ref{prop:nonvanish}}
\label{sec:finish_G}

If we assume \eqref{eqn:zeros_trapped}, then the lower bound in Proposition~\ref{prop:lower_bound_high_deriv} and the upper bound Proposition~\ref{prop:upper_bound_high_deriv} combine as follows:  There exist constants $c_{30}$ and $c_{31}$ such that if $K\geq 1600\eta\mathcal{L}+c_{30}$, then there exists an integer $k\in[K,2K+1]$ such that
\[
\frac{1}{2(100)^{k+1}}\leq \eta^2 \int_{N_{\eta}}^{N_{\eta}^*}\Big|\sum_{\mathrm{N}\mathfrak{p}\in[N_{\eta},u]}\frac{\lambda_{\pi_{\eta}}(\mathfrak{p})(\lambda_{\pi'}(\mathfrak{p})+\lambda_{\widetilde{\pi}}(\mathfrak{p}))\log\mathrm{N}\mathfrak{p}}{\mathrm{N}\mathfrak{p}}\Big|\frac{du}{u}+\frac{c_{31}k}{110^k}.
\]
If $c_{30}$ in \eqref{eqn:K_pre} is further enlarged with respect to $c_{31}$, then we ensure that
\begin{align*}
1&\ll_{\eta} 100^{2K}\int_{N_{\eta}}^{N_{\eta}^*}\Big|\sum_{\mathrm{N}\mathfrak{p}\in[N_{\eta},u]}\frac{\lambda_{\pi_{\eta}}(\mathfrak{p})(\lambda_{\pi'}(\mathfrak{p})+\lambda_{\widetilde{\pi}}(\mathfrak{p}))\log\mathrm{N}\mathfrak{p}}{\mathrm{N}\mathfrak{p}}\Big|\frac{du}{u}\\
&\ll_{\eta} 100^{2K}\sup_{u\in[N_{\eta},N_{\eta}^{*}]}\Big|\sum_{\mathrm{N}\mathfrak{p}\in[N_{\eta},u]}\frac{\lambda_{\pi_{\eta}}(\mathfrak{p})(\lambda_{\pi'}(\mathfrak{p})+\lambda_{\widetilde{\pi}}(\mathfrak{p}))\log\mathrm{N}\mathfrak{p}}{\mathrm{N}\mathfrak{p}}\Big|\int_{N_{\eta}}^{N_{\eta}^*}\frac{dt}{t}\\
&\ll_{\eta} 100^{2K}K\sup_{u\in[N_{\eta},N_{\eta}^{*}]}\Big|\sum_{\mathrm{N}\mathfrak{p}\in[N_{\eta},u]}\frac{\lambda_{\pi_{\eta}}(\mathfrak{p})(\lambda_{\pi'}(\mathfrak{p})+\lambda_{\widetilde{\pi}}(\mathfrak{p}))\log\mathrm{N}\mathfrak{p}}{\mathrm{N}\mathfrak{p}}\Big|.
\end{align*}
By \eqref{eqn:separate_dirichlet_coeffs_2} and \eqref{eqn:Neta_def}, if $\mathrm{N}\mathfrak{p}\geq N_{\eta}$, then $|\lambda_{\pi_{\eta}}(\mathfrak{p})|^2\log\mathrm{N}\mathfrak{p} = \Lambda_{\pi_{\eta}\times\widetilde{\pi}_{\eta}}(\mathfrak{p})$.  
Therefore, upon squaring both sides, the Cauchy--Schwarz inequality and Lemma~\ref{lem:nonneg}  yield
\begin{align*}
1&\ll_{\eta} 100^{4K}K^2\Big(\sum_{\mathrm{N}\mathfrak{p}\in[N_{\eta},N_{\eta}^*]}\frac{|\lambda_{\pi_{\eta}}(\mathfrak{p})|^2\log\mathrm{N}\mathfrak{p}}{\mathrm{N}\mathfrak{p}}\Big)\Big(\sum_{\mathrm{N}\mathfrak{p}\in[N_{\eta},N_{\eta}^*]}\frac{|\lambda_{\pi'}(\mathfrak{p})+\lambda_{\widetilde{\pi}}(\mathfrak{p})|^2\log\mathrm{N}\mathfrak{p}}{\mathrm{N}\mathfrak{p}}\Big)\\
&\ll_{\eta} 100^{4K}K^2\Big(\sum_{\mathfrak{a}}\frac{\Lambda_{\pi_{\eta}\times\widetilde{\pi}_{\eta}}(\mathfrak{a})}{\mathrm{N}\mathfrak{a}^{1+1/\log(eN_{\eta}^*)}}\Big)\Big(\sum_{\mathrm{N}\mathfrak{p}\in[N_{\eta},N_{\eta}^*]}\frac{|\lambda_{\pi'}(\mathfrak{p})+\lambda_{\widetilde{\pi}}(\mathfrak{p})|^2\log\mathrm{N}\mathfrak{p}}{\mathrm{N}\mathfrak{p}}\Big).
\end{align*}

The sum over $\mathfrak{a}$ is $\ll \eta^{-1}+\log N_{\eta}\ll_{\eta} K$ by Lemma~\ref{lem:Mertens}.  For the sum over $\mathfrak{p}$, we use Proposition~\ref{prop:Linnik} to conclude that there exists a constant $c_{32}$ such that if $K$ in \eqref{eqn:K_pre} equals $\lceil c_{32}\eta\mathcal{L}+c_{30}\rceil$, then $1\ll_{\eta} (1-\beta_1)100^{4K}K^6\ll_{\eta} (1-\beta_1)\mathfrak{C}_{\pi}^{148c_{32}n_0^2\eta}$.  Proposition~\ref{prop:nonvanish} follows once we choose $\eta=\varepsilon/(148c_{32}n_0^2)$.

\section{Finishing the proof of Theorem~\ref{thm:main}}
\label{sec:Proof}

Since $\pi'$ is fixed, we see that if $\pi'^*=\widetilde{\pi}^*$, then Theorem~\ref{thm:main} follows from \eqref{eqn:HarcosThorner_RS} (with $\chi=|\cdot|^{it_{\pi}}$) and Lemma~\ref{lem:t_pi_bound}.  If $\mathfrak{C}_{\pi}\leq\mathfrak{C}_{\pi'}$, then Theorem~\ref{thm:main} follows from \eqref{eqn:Brumley_RS}, since $\pi'$ is fixed.   Therefore, we may assume that $L(s,\pi\times\pi')$ is entire and $\mathfrak{C}_{\pi'}\leq\mathfrak{C}_{\pi}$.

\begin{lemma}
\label{lem:logderiv_asymptotic}
Let $\pi\in\mathfrak{F}_n$ and $\pi'\in\mathfrak{F}_{n'}$.  If $\widetilde{\pi}^*\neq\pi'^*$, $\mathfrak{C}_{\pi}\geq\mathfrak{C}_{\pi'}$, $\sigma_1\in[1,2]$, $x\geq \exp(\mathfrak{C}_{\pi}^{2\varepsilon})$, and $0<\varepsilon<1/(10n[F:\mathbb{Q}])$, then
\begin{equation}
\label{eqn:lem6.1}
\sum_{\mathrm{N}\mathfrak{a}\leq x}\frac{\Lambda_{\pi\times\pi'}(\mathfrak{a})}{\mathrm{N}\mathfrak{a}^{\sigma_1}}\Big(1-\Big(\frac{\mathrm{N}\mathfrak{a}}{x}\Big)^2\Big)=-\frac{L'}{L}(\sigma_1,\pi\times\pi')+O_{\pi',\varepsilon}(1).
\end{equation}
\end{lemma}
\begin{proof}
Applying Proposition~\ref{prop:nonvanish} to $(\pi[it],\pi')\in\mathfrak{F}_n\times\mathfrak{F}_{n'}$ and rescaling $\varepsilon$ using \eqref{eqn:BH}, we obtain the zero-free region in \eqref{eqn:t-shifted}.  Without loss of generality, we may assume that $0<c_{10}\leq 1$.  Mellin inversion gives
\begin{equation}
\label{eqn:xiannan_contour_push_1}
\sum_{\mathrm{N}\mathfrak{a}\leq x}\frac{\Lambda_{\pi\times\pi'}(\mathfrak{a})}{\mathrm{N}\mathfrak{a}^{\sigma_1}}\Big(1-\Big(\frac{\mathrm{N}\mathfrak{a}}{x}\Big)^2\Big)=-\frac{1}{2\pi i}\int_{3-i\infty}^{3+i\infty}\frac{L'}{L}(s+\sigma_1,\pi\times\pi')\frac{2x^s}{s(s+2)}ds.
\end{equation}
For $t\in\mathbb{R}$, consider the parametric curve $\mathscr{C}(t)=1-\sigma_1-(c_{10}/2)(\mathfrak{C}_{\pi}(|t|+3))^{-\varepsilon}+it$.   By Lemma~\ref{lem:log_deriv_strip} and \eqref{eqn:t-shifted}, in the region $\mathrm{Re}(s)\geq \mathrm{Re}(\mathscr{C}(\mathrm{Im}(s)))$, $-\frac{L'}{L}(s+\sigma_1,\pi\times\pi')$ is holomorphic and bounded by $O_{\pi',\varepsilon}((\mathfrak{C}_{\pi}(|t|+3))^{2\varepsilon})$ with an ineffective implied constant.  Since $\mathscr{C}'(t)\ll_{\varepsilon}1$ for $t\neq 0$, we deform the line of integration to $\mathscr{C}$, picking up a residue at $s=0$.  Therefore,
\begin{equation*}
\begin{aligned}
\sum_{\mathrm{N}\mathfrak{a}\leq x}\frac{\Lambda_{\pi\times\pi'}(\mathfrak{a})}{\mathrm{N}\mathfrak{a}^{\sigma_1}}\Big(1-\Big(\frac{\mathrm{N}\mathfrak{a}}{x}\Big)^2\Big)&=-\frac{L'}{L}(\sigma_1,\pi\times\pi')+\frac{1}{2\pi i}\int_{\mathscr{C}}-\frac{L'}{L}(s+\sigma_1,\pi\times\pi')\frac{2x^s}{s(s+2)}ds\\
&=-\frac{L'}{L}(\sigma_1,\pi\times\pi')+O_{\pi',\varepsilon}\Big(\mathfrak{C}_{\pi}^{3\varepsilon}x^{1-\sigma_1}\int_{0}^{\infty}\frac{(t+3)^{2\varepsilon-2}}{x^{c_{10}/(2(\mathfrak{C}_{\pi}(t+3))^{\varepsilon})}}dt\Big).
\end{aligned}
\end{equation*}
Via the change of variables $u=\frac{1}{2}c_{10}\mathfrak{C}_{\pi}^{-\varepsilon}(t+3)^{-\varepsilon}\log x$, the error term is $O_{\pi',\varepsilon}(1)$.
\end{proof}

By \eqref{eqn:LRS_2}, the lower bound in Theorem~\ref{thm:main} is trivial when $\sigma>2$.  Let $0<\varepsilon<\frac{1}{10n[F:\mathbb{Q}]}$ and $0<\varpi<\min\{\frac{1}{10n[F:\mathbb{Q}]},\frac{\varepsilon}{18(n+n')}\}$.  If $x=\exp(\mathfrak{C}_{\pi}^{2\varpi})$ and $\sigma\in[1,2]$, then we integrate both sides \eqref{eqn:lem6.1} over $\sigma_1\in[\sigma,2]$.  By \eqref{eqn:Brumley_CS} and the Cauchy--Schwarz inequality, we observe that
\begin{align*}
|\log L(\sigma,\pi\times\pi')|
&\leq \Big(\sum_{2\leq \mathrm{N}\mathfrak{a}\leq x}\frac{\Lambda_{\pi\times\widetilde{\pi}}(\mathfrak{a})}{\mathrm{N}\mathfrak{a}^{\sigma}\log\mathrm{N}\mathfrak{a}}\Big)^{\frac{1}{2}}\Big(\sum_{2\leq \mathrm{N}\mathfrak{a}\leq x}\frac{\Lambda_{\pi'\times\widetilde{\pi}'}(\mathfrak{a})}{\mathrm{N}\mathfrak{a}^{\sigma}\log\mathrm{N}\mathfrak{a}}\Big)^{\frac{1}{2}}+O_{\pi',\varpi}(1)\\
&\leq e(\log L(\sigma+\tfrac{1}{\log(ex)},\pi\times\widetilde{\pi}))^{\frac{1}{2}}(\log L(\sigma+\tfrac{1}{\log(ex)},\pi'\times\widetilde{\pi}'))^{\frac{1}{2}}+O_{\pi',\varpi}(1).
\end{align*}
Since $\mathfrak{C}_{\pi'}\leq\mathfrak{C}_{\pi}$, the last line is $\leq 9(n+n')\varpi\log\mathfrak{C}_{\pi}+O_{\pi',\varpi}(1)$ by \eqref{eqn:BH}, Lemma~\ref{lem:Li1}, and our choice of $x$.  Thus, there exists an ineffective constant $c_{33}=c_{33}(n,F,\pi',\varpi)>0$ such that if $\pi\in\mathfrak{F}_n$ and $\sigma\in[1,2]$, then $|L(\sigma,\pi\times\pi')|\geq c_{33}\mathfrak{C}_{\pi}^{-9(n+n')\varpi}\geq c_{33}\mathfrak{C}_{\pi}^{-\varepsilon/2}$.

Suppose now that $c_{34}>0$ is a constant.  When $\sigma\in[1-c_{34}\mathfrak{C}_{\pi}^{-\varepsilon},1)$, the theory of line integrals and Lemma~\ref{lem:Li1} imply that
\[
|L(\sigma,\pi\times\pi')-L(1,\pi\times\pi')|\leq (1-\sigma)\sup_{w\in[\sigma,1]}|L'(w,\pi\times\pi')|\ll_{\pi',\varepsilon}c_{34}\mathfrak{C}_{\pi}^{-7\varepsilon/8}.
\]
If $c_{34}$ is suitably small in terms of $\pi'$, $n$, $F$, and $\varepsilon$, then by the lower bound on $L(1,\pi\times\pi')$ proved above, Theorem~\ref{thm:main} follows in all cases.


\bibliographystyle{abbrv}
\bibliography{JAThorner_ZeroFreeRegion}
\end{document}